\documentclass[a4paper,10pt]{amsart}
\usepackage[english]{babel}
\usepackage{amsthm, amsmath, amssymb}
\usepackage{amsfonts} 

\usepackage{verbatim}

\usepackage{graphicx}

\newsavebox\verbbox

\numberwithin{equation}{section}
\usepackage{latexsym}
\usepackage{amscd}
\usepackage{enumerate,mathrsfs,url}
\usepackage[usenames,dvipsnames]{color}
\usepackage[normalem]{ulem}

\newcommand{\un}{\mathbf{1}}

\newcommand{\N}{\mathbb{N}} 
\newcommand{\R}{\mathbb{R}} 
\newcommand{\Z}{\mathbb{Z}} 
\newcommand{\X}{\mathcal{X}}

\newcommand{\egaldef}{:=} 

\newcommand{\paren}[1]{\left( \left. #1 \right. \right)}
\newcommand{\parenb}[1]{\bigl( #1 \bigr)}
\newcommand{\parenB}[1]{\Bigl( #1 \Bigr)}

\newcommand{\croch}[1]{\left[ \left. #1 \right. \right]}

\newcommand{\set}[1]{\left\{ \left. #1 \right. \right\}}

\newcommand{\absj}[1]{\left\lvert #1 \right\rvert} 

\newcommand{\Proba}{\mathbb{P}}
\newcommand{\moy}[1]{\left\langle #1 \right\rangle}
\newcommand{\moyb}[1]{\bigl\langle #1 \bigr\rangle}
\newcommand{\moyB}[1]{\Bigl\langle #1 \Bigr\rangle}
\DeclareMathOperator{\sign}{sign} 
\DeclareMathOperator{\card}{Card} 

\newcommand{\cK}{\mathcal{K}}
\newcommand{\cN}{\mathcal{N}}

\newcommand{\HH}{SP}

\newcommand{\rhowinter}{\rho}
\newcommand{\rhoseason}{m_{\rho}}
\newcommand{\rhoseasondiscrete}[1]{m_{\rho,#1}}
\newcommand{\rhomax}{\ensuremath{\mu_{\max}}} 
\newcommand{\survival}{\ensuremath{\mathbb{S}}}

\newtheorem{proposition}{Proposition}[section]

\newtheorem{theorem}[proposition]{Theorem}

\newtheorem{corollary}[proposition]{Corollary}

\theoremstyle{remark}
\newtheorem{remark}[proposition]{Remark}

\theoremstyle{definition}
\newtheorem{definition}[proposition]{Definition}

\newcommand{\D}{\mathcal{D}}

\newcommand{\K}{\mathcal{K}}

\newcommand{\esssup}{\operatornamewithlimits{ess\,sup}} 
\newcommand{\essinf}{\operatornamewithlimits{ess\,inf}} 

\newcommand{\JA}{\HH$_1$}

\begin{document}

\author[S. Arlot, S. Marmi, D. Papini]{Sylvain Arlot, Stefano Marmi, Duccio Papini}

\title[Chaotic livestock commodities market cycles]{Coupling the Yoccoz-Birkeland population model with price dynamics: chaotic livestock commodities market cycles}
\date{\today \quad Preliminary version}

\address{Laboratoire de Math\'ematiques d'Orsay, Univ. Paris-Sud, CNRS, Universit\'e Paris-Saclay, 91405 Orsay, France; 
Select Project-Team,  Inria Saclay - Ile de France}
\email{sylvain.arlot@u-psud.fr}

\address{Scuola Normale Superiore,
Piazza dei Cavalieri,
7 - 56126 Pisa, Italy}
\email{stefano.marmi@sns.it}

\address{Universit\`a degli Studi di Udine, Dipartimento di Scienze Matematiche, Informatiche e Fisiche,
via delle Scienze 206, 33100, Udine, Italy}
\email{duccio.papini@uniud.it}

\begin{abstract}
We propose a new model for the time evolution of livestock commodities prices
which exhibits endogenous deterministic stochastic behaviour.
The model is based on the Yoccoz-Birkeland integral equation, a model first
developed for studying the time-evolution of
single species with high average  fertility, a relatively short mating season and 
density-dependent reproduction rates. 
This equation is then coupled with 
a differential equation describing the price of a livestock commodity driven by the unbalance between its demand and supply.
At its birth the cattle population is split into two parts: reproducing females and cattle for butchery. The relative amount of the two is determined 
by the spot price of the meat. 
We prove the existence of an attractor and we investigate numerically its properties: the strange attractor existing for the 
original Yoccoz-Birkeland model is persistent but its chaotic behaviour depends also on the time evolution of the price in an essential way. 
\end{abstract}

\thanks{%
This work was started during a visit of SA to the Scuola Normale
Superiore in Pisa and to the Complex Systems Center of the University of Siena, 
the hospitality  of which is gratefully acknowledged.
The authors acknowledge the support of the Centro di Ricerca Matematica Ennio de Giorgi
and of UniCredit Bank R\&D group for financial support through the "Dynamics and Information Theory Institute"
at the Scuola Normale Superiore.
Most of this work was done while SA was a researcher at CNRS 
(D\'epartement d'Informatique / \'Ecole Normale Sup\'erieure, UMR 8548 CNRS/\-ENS/\-INRIA), 45, rue d'Ulm, F-75230 PARIS Cedex 05, France). 
DP was supported by the INdAM-GNAMPA project ``Propriet\`a qualitative di alcuni problemi ai limiti''.
We are also grateful to Paolo Nistri for contributing to the initial formulation of the model 
and to Nicola Zanda for useful discussions 
on livestocks farming (in particular Cinta Senese and pig). 
}

\maketitle

\section{Introduction}
\label{sec.intro}
Twenty years ago a new model of population dynamics which exhibits endogenous chaotic behaviour has been 
proposed by J.-C. Yoccoz and H. Birkeland \cite{Yoc_Bir:1998}, 
(we refer to \cite{AMMY:2018} for a historical perspective 
on this work).
The model was prompted from the evidence of aperiodic large oscillations (2-3 orders of magnitude on a 3-5 years time span)
in the time evolution of the population of a species of rodents, Microtus Epiroticus (sibling vole) on Svalbard Islands. 
This species has a high fertility rate which has a strong dependence on seasonal factors 
(due to harsh weather conditions in winter)
and on the population density. 
Indeed, few good reproduction spots are available and their quality decreases as the population increases. 
Despite the absence of any significant predator and the relative abundance of food, one observes high oscillations of population. 

The Yoccoz-Birkeland model was studied through a mathematical analysis 
and some simulation experiments in \cite{Arl:2004,Nie_Pac_Vie:2012}. 
In short, it has been showed that such a deterministic model can produce complex dynamics 
with a high sensitivity to initial conditions,
only by the combination of density-dependent fertility, the lag due to the maturation age and a periodic seasonality. 
A detailed account of these results is provided in Section~\ref{sec.modelYB}. 
\medskip

In this paper, we introduce a new model coupling the population dynamics in the Yoccoz-Birkeland model 
with an equation modelling the price dynamics of a livestock commodity market inspired by \cite{Bel_Mac:1989}. 
A cattle population is split at the birth into reproducing females and cattle for butchery. 
The relative amount of the two is determined by the spot price of the meat whereas the logarithmic derivative of the price is 
determined by the unbalance between the demand and supply of the meat.  
On the population side, seasons (or artificial synchronization of births) 
and maturation lags are also taken explicitely into account.

The importance of the livestock commodities market in economics is related to the fact that it provides some of the oldest 
and best documented examples of business cycles. 
Approximately periodic fluctuations of supply and prices were
first observed at the beginning of last century in hog markets 
\cite{Hanau:1928}, \cite{Eze:1938} and since then they have been the object of many studies.  
Beef cattle stocks ``are among the most periodic time series in economics'' \cite{Ros_Mur_Sch:1994}, 
a fact also related to the double role played by cattle as capital as well as consumption goods.

In many respects, the continuing presence of any price cycle is disturbing: 
if a predictable price cycle exists, then producers responding in a countercyclical fashion 
could earn larger than ``normal'' profits over time \cite{Hay_Sch:1987}. 
Such profits could occur even with lags in the production process (substantial gestation
and maturation times stretch it over lengthy
intervals of time) because predictable price movements would still influence production decisions. 
Eventually, countercyclical production response would smooth out price fluctuations at the market level, 
causing the cycle to disappear.

An alternative explanation for the existence of a business cycle is that the cycle itself is not perfectly predictable: 
the law of motion may be a deterministic nonlinear relationship that generates unpredictable patterns \cite{Gra_Mal:1986}). 
Cobweb models show that complicated price dynamics may indeed occur due to nonlinearities \cite{Chi:1988, Hommes:1994}
and that simple expectation rules in a nonlinear environment may lead to chaotic price fluctuations
\cite{Hommes:2013}. 

Our model shows how, under quite natural assumptions, simply connecting the percentage 
of reproducing females with the price equation gives 
rise to a chaotic time evolution of price. 
This is characterized by a series of ``cycles'' of booms and busts (i.e. rapid increase or decrease).
Another important feature of the model is that it takes into account some specific characteristics 
of the production of livestock
commodities often neglected in the literature: 
for example the existence of time-lags between the producer decision on the reproduction
strategy and the butchery of the calves, the synchronization of births and their seasonality.
We point out that some reproductive constraints included in Yoccoz-Birkeland model 
are realized also in the production of livestock
commodities, besides synchronization of births and seasonality.
For example, the development of breeding facilities requires time and limits the reproductive capacity.
\medskip

Here follows the plan of our paper.
In Section~\ref{sec.modelYB} we review the model proposed by Yoccoz and Birkeland and the relative results that were shown in
\cite{Arl:2004,Nie_Pac_Vie:2012}.

Our model is derived in Section~\ref{sec.our-model} by coupling in a suitable way the Yoccoz-Birkeland model with a differential
equation for the logarithmic derivative of the price which is inspired by \cite{Bel_Mac:1989}.

In Section~\ref{sec.math} our model is rigorously analyzed.
Global existence and uniqueness for initial value problems are obtained along with some useful estimates on the solutions.
In particular we show that global boundedness and persistence follows 
under suitable assumptions involving some relevant biological parameters.
Finally we prove the existence of a global attractor containing at least a non-trivial periodic solution.

In Section~\ref{sec.simus} numerical experiments show this new model 
can produce complex dynamics in the population size and the price.
The main setting of the parameters is chosen having in mind the attractor detected in \cite{Arl:2004} for Yoccoz-Birkeland model,
on the one hand, and the hog market, on the other.
The attractor we found has both sensitive dependence on initial conditions and noninteger dimension.
The relevance of the presence of the market dynamics, based on the unbalance of demand and supply, 
is outlined in a second numerical experiment
in which the population dynamics is suitably decoupled from the price evolution 
and gives rise to a behavior which looks like by no means
chaotic, but asymptotically periodic, in fact.
Bifurcation diagrams shows that complex dynamics persists for realistic values of the maximal fertility 
and also for weaker levels of
dependence of the fertility rate on the total population.
Technical details of the numerics are given in the Appendix.

\section{The Yoccoz-Birkeland model}
\label{sec.modelYB}

This section recalls the model proposed in \cite{Yoc_Bir:1998} 
and gives a brief summary of results obtained in \cite{Arl:2004,Nie_Pac_Vie:2012}. 

The Yoccoz-Birkeland model aims at modelling the population of mature females of a single species 
with density-dependent reproduction
rate and whose reproduction strategy may be influenced by seasons or other external factors.

More precisely, the model proposed by Yoccoz and Birkeland goes as follows:
\begin{itemize}
\item $t$ is the time measured in years;
\item $A_0$ is the age (in years) at which females reach sexual maturity;
\item $A_1$ is the maximum age (in years) for females;
\item $N(t)$ is the number of sexually mature females at time $t$;
\item $m(N)$ is the density-dependent female reproduction rate and measures the average number of female cubs that a single female can give birth
to in a year in optimal weather conditions when the total number of mature females is $N$; it is reasonably a decreasing function of $N$;
\item $\rhoseason(t)$ is the seasonal factor and gives the fraction of females actually reproducing at time $t$; typically it is $1$-periodic;
\item $\survival(a)$ is the fraction of newborn females still alive at age $a$ (in years); 
\end{itemize}
then the number of females with age ranging in $[a,a+da]$ is
given by
\[
N(t-a)m \bigl( N(t-a) \bigr) \rhoseason(t-a)\survival(a)da \enspace .
\]
Therefore, $N$ satisfies the following integral equation
\begin{equation}\label{eq.YoccozBirkeland}
N(t)=\int_{A_0}^{A_1}N(t-a)m \bigl( N(t-a) \bigr) \rhoseason(t-a)\survival(a)da \enspace ,
\end{equation}
which allows to uniquely determine $N(t)$ for $t\in[t_0,t_0+A_0]$ (and for every other $t>t_0$ by recursion)
if $N(t)$ is known for $t\in[t_0-A_1,t_0]$.

This model have been proposed to explain the behaviour of the population of Microtus Epiroticus in Svalbard Isles which has high average fertility, while its numbers show large fluctuations so that sometimes the species looks to be even close to extinction in spite of the absence of any significant predation \cite{Yoc_Ims:1999}.
In the case of this species of small rodents the biological explanation of the observed behavior relies on the following remarks.
First, the adverseness of the environment in which the population lives causes a shortage of good reproduction spots and, as a consequence, the larger is the number of sexually active females (above some threshold) the smaller becomes the average individual fertility because the spots become overcrowded.
Secondly, the shortness of the reproduction season (summer) induces a decrease of the age at which females reach sexual maturity.
In this way, the females that are born at the beginning of summer are able to give birth to their first cubs before the end of the same summer.
These facts translate into a density-dependent reproduction rate that decreases as the population increases and into a sexual maturation age that is less than the average length of summer.

In \cite{Arl:2004} the following choices were made: 
\begin{align}
\survival(a)&=1-\frac{a}{A_1}\,,\quad\text{for }a\in[0,A_1]\nonumber\\
\rhoseason(t)&=
\begin{cases}
0 &\text{if }0\le t< \rhowinter \text{ (mod 1)}\\
1 &\text{if }\rhowinter \le t<1 \text{ (mod 1)}
\end{cases}\label{eq.mrho.Arl:2004}\\
m(N)     &=
\begin{cases}
m_0            &\text{if }N\le 1\\
m_0N^{-\gamma} &\text{if }N>1
\end{cases}\label{eq.density.dependent.fertility}
\end{align}
where $ \rhowinter \in(0,1)$ stands for the average length of winter, $m_0$ is the average yearly female fertility
in optimal weather and environmental conditions and $\gamma\ge 1$. 

Equation \eqref{eq.YoccozBirkeland} gives rise to a continuous semi-group as follows.
For each $t_0\in\R/\Z$ set
\[
Y_{t_0}=\left\{N\in C \bigl( [-A_1,0] \bigr) : N(0)=\int_{A_0}^{A_1}N(-a)m \bigl( N(-a) \bigr) \rhoseason(t_0-a)\survival(a)da\right\}
\]
and consider the phase space
\[
Y^\sharp=\{(t,N):t\in\R/\Z, N\in Y_t\}
\]
which is a complete metric space with respect to the distance $d((s,M),(t,N))=|s-t|_{\R/\Z}+\|M-N\|_\infty$.
Then the semi-group $(T^s)_{s\ge 0}$ generated by \eqref{eq.YoccozBirkeland} is given by
$ T^s(t,N) =(t+s(\operatorname{mod}1),N_t^s) $ where:
\[
N_t^s(-a) =
\begin{cases}
 N(s-a)                                                                           &\hspace{-1.5cm}\text{if }0\le s\le a\le A_1 \\
\displaystyle\int_{A_0}^{A_1} \!\!\! N(s-a-b)m \bigl( N(s-a-b) \bigr) \rhoseason(t+s-a-b)\survival(b)db &\text{otherwise.}
\end{cases}
\]
The following result holds.
\begin{theorem}[\cite{Arl:2004}]\label{thm.Arlot2004}
Assume that $N \mapsto N m(N)$ is uniformly continuous on $[0,+\infty)$ and that 
\begin{enumerate}
\item $m_0/2 \le m(N)\le m_0$ if $N \le 1$ and $\min\{1/2, N^{-\gamma}\} m_0 \le m(N)\le m_0 N^{-\gamma}$ if $N \ge 1$, $\gamma \ge 1$;
\item $0 \le \rhoseason(t) \le 1$ for all $t$ and $\rhoseason(t) = 1$ on an interval of length $1 - \rhowinter > 0$, $\rhowinter > 0$; 
\item $A_1 \ge \max\{2A_0, A_0+1\}$ and $c_0 m_0 > 2$ where $c_0 = \int_{A_0 + \rhowinter}^{A_0 + 1} \survival(a) da$.
\end{enumerate}
Moreover, let
\[
N_{\max} = m_0 \frac{A_1}{2} \left(1 - \frac{A_0}{A_1}\right) \quad \text{and} \quad L = m_0 \left(3 - \frac{A_0}{A_1}\right)
\]
and consider the set
\[
\begin{aligned}
\mathcal{K}=\left\{(t_0,N)\in Y^\sharp:\vphantom{\int}\right.&\frac{c_0m_0}{2}N_{\max}^{1-\gamma}\le N(s)\le N_{\max} 
\ \forall s\in[-A_1,0]\text{ and } 
\\
&\left.\vphantom{\int} \bigl\lvert N(s_1)-N(s_2) \bigr\rvert \le L \lvert s_1-s_2 \rvert
\ \forall s_1,s_2\in[-A_1,0]\right\}.
\end{aligned}
\]
Then
\begin{enumerate}
\item $\mathcal{K}$ is compact and $T^s(t_0,N)\in\mathcal{K}$ for all $s\ge 0$ and all $(t_0,N)\in\mathcal{K}$;
\item for each $(t_0,N)\in Y^\sharp$ there exists $s_0\ge 0$ such that $T^s(t_0,N)\in\mathcal{K}$ for all $s\ge s_0$;
\item the compact set $\Lambda=\bigcap_{n\ge 0}T^n(\mathcal{K})$ is invariant: $T^s(\Lambda)=\Lambda$ for all $s\ge 0$;
\item for each neighbourhood $U$ of $\Lambda$ and each $(t_0,N)\in Y^\sharp$ there exists $s_0\ge 0$ such that $T^s(t_0,N)\in U$ for all $s\ge s_0$;
\item $\Lambda$ is a global attractor of $\left((T^s)_{s\ge 0},Y^\sharp\right)$.
\end{enumerate}
\end{theorem}
Simulations were also performed in \cite{Arl:2004} on a discretization of \eqref{eq.YoccozBirkeland} 
with a smoothed version of \eqref{eq.mrho.Arl:2004} and \eqref{eq.density.dependent.fertility} 
for $\rhoseason(t)$ and $m(N)$, and the choices $m_0 = 50$ and $A_1 = 2$. 
The parameter space was explored with respect to $\gamma \in [2, 16]$, $\rhowinter \in [0, 0.5]$, $A_0 \in [0, 0.4]$.
The computations showed the existence of periodic points, of possible Hopf bifurcations, 
the coexistence of different attractors and the presence of complex dynamics.
In particular, a complex attractor is outlined in the cases $\rhowinter=0.30$, $\gamma=8.25$ and $A_0=0.15$ or $A_0=0.18$; 
a detailed study of its dynamical features is done for $A_0=0.15$.
In any case, the analysis and the simulations in \cite{Arl:2004} show that the Yoccoz-Birkeland model 
recovers to some extent the general behavior and the biological characteristics of Microtus Epiroticus outlined above.
Numerical solutions of \eqref{eq.YoccozBirkeland} have large oscillations, with minima close to extinction, 
thanks to the interaction of the density-dependent fertility, 
the relatively quick sexual maturation of females and the average duration of winter.

Recently, the Yoccoz-Birkeland model has been the subject of the paper \cite{Nie_Pac_Vie:2012}.
There the analogue of Theorem~\ref{thm.Arlot2004} and the existence of periodic points have been proved 
for the discrete version of the model and
numerical simulations have been done with special emphasis to the case with $A_0=0.18$.

\section{A model coupling market and population dynamics}
\label{sec.our-model}

This section presents the new model proposed in this paper for cattle population and price dynamics. 
The idea is to couple a population dynamics model similar to the Yoccoz-Birkeland with a market dynamics model.
The (cattle) population is split into two parts: on the one hand, females for reproduction; on the other hand, cattle for butchery (all the males plus some of the females).

The mechanism is the following:
\begin{itemize}
\item At the birth of some babies, part of the newborn females are put in the reproduction line, and the remaining newborn females are put in the butchery line together with all newborn males. 
The fraction $R$ of newborn females that will reproduce (chosen by the breeder) is only determined by the price of meat at birth time.
\item In the reproducing line, females have children between ages $A_0$ and $A_1$.
Their fertility can be affected by seasons, or because births are synchronized by the breeder (through a function $\rhoseason(t)$).
Reproducing females older than $A_1$ (hence, non fertile) are not taken into account anywhere in the model.
\item In the butchery line, cattle can be butchered between ages $\Omega_0$ and $\Omega_1$. So, only the (alive) butchery population between ages $\Omega_0$ and $\Omega_1$ can count as a ``supply'' for the market. 
\item The price evolution is a simple function of the supply (which comes from the butchery line population dynamics) and the demand (which depends only on the price).
\end{itemize}

Note that contrary to the Yoccoz-Birkeland model, we assume no mortality before ages $A_1$ (resp. $\Omega_1$). 

\subsection{Notation and parameters}

\begin{itemize}
\item $t$ is the time measured in years.
%
%
\item $N_r(t)$ is the total population of {\em mature females} that are in the reproducing line 
and can give birth to pups at time $t$.
\item $N_{b}(t)$ is the total population of cattle that is {\em suitable for butchery} at time $t$ 
(both males and non-reproducing females, old enough and in the butchery line).
\item $R(P)$ is the fraction of newborn females that are put in the reproducing line 
when the price of meat is $P$ when they are born.
\item $A_0$ is the age from which females can have children (i.e., age of sexual maturity + length of the first gestation).
\item $A_1$ is the maximal age at which females can give birth to children 
(i.e., age of sexual unfertility + length of the last gestation).
\item $\Omega_0$ is the minimal age at which the cattle (male or female) can be butchered.
\item $\Omega_1$ is the maximal age at which the cattle (male or female) can be butchered. 
Note that $\Omega_1$ could possibly be enlarged (compared to its biological value) if the meat can be frozen after butchering.
\item $m(N)$ is the average annual female (resp. male) fertility of each mature female when the total population is $N$, 
i.e., the average number of female (resp. male) babies per year for a single mature female. 
Typically it is a decreasing function as in \eqref{eq.density.dependent.fertility}. 
We assume the sex ratio is $1/2$, i.e., the average number of male babies is equal to the average number of female babies 
(hence, $m(N)$ is half of the average annual fertility).
\item $\rhoseason(t)$ is the 1-periodic step function (with $\int_0^1 \rhoseason(t)dt=1$) 
that accounts for a possible modulation of fertility during each year (births synchronization or seasonal effects).
%
%
\item $P(t)$ is the market price of meat at time $t$.
\item $D(P)$ is the demand of the market (per time unit) when the price of meat is $P$ (typically a decreasing function of $P$).
\item $S(t)$ is the supply to the market (per time unit) at time $t$ (typically proportional to $N_b(t)$).
\item $\lambda$ is a ``temperature'' parameter of the meat market: 
higher values of $\lambda$ correspond to bigger price variations
in response to the same demand/supply imbalance.
\item $F(D,S)$ is the function of demand and supply that rules the meat price dynamics.
\end{itemize}

\subsection{Population dynamics}
The population dynamics model is strongly inspired from the Yoccoz-Birkeland model \eqref{eq.YoccozBirkeland}.
In order to derive the equations satisfied by $N_r$ and $N_b$, let us define the following additional notation: 

\begin{itemize}
\item $B_f(t)$ is the density of newborn female cattle at time $t$ (i.e., $B_f(t)dt$ females are born between $t$ and $t+dt$).
\item $B_m(t)$ is the density of newborn male cattle at time $t$.
\item $B_r(t)$ is the density of newborn (female) cattle that are put in the reproducing line at time $t$.
\item $B_b(t)$ is the density of newborn cattle that are put in the butchery line at time $t$.
\end{itemize}

First, the male and female birth densities at time $t$ are given by:
\begin{equation}
\label{eq.birth}
B_f(t) = B_m(t) = \rhoseason(t) m \bigl( N_r(t) \bigr) N_r(t) 
\enspace . 
\end{equation}

The breeder decides at birth time $t$ which fraction $R(P(t))$ of the newborn females is going into the reproducing line, which gives:
\begin{align}
\label{eq.birth.reproduction}
B_r(t) &= B_f(t) R \bigl( P(t) \bigr) 
= \rhoseason(t) m \bigl(N_r(t) \bigr) N_r(t) R \bigl( P(t) \bigr) \\
\mbox{and} \qquad \notag
B_b(t) &= B_m(t) + B_f(t)  \bigl[ 1 - R  \bigl( P(t) \bigr) \bigr] \\
&= \rhoseason(t) m \bigl( N_r(t) \bigr) N_r(t) \bigl[ 2 - R  \bigl( P(t)  \bigr) \bigr] 
 \enspace .
\label{eq.birth.butchery}
\end{align}

We assume no mortality at all between birth and the end of reproduction time for females, 
so the number of mature females is given by
\begin{equation}
\label{eq.mature.females}
N_r(t) = \int_{A_0}^{A_1}  B_r(t-a) da 
= \int_{A_0}^{A_1} \rhoseason(t-a) m \bigl( N_r(t-a) \bigr) N_r(t-a) R \bigl( P(t-a)  \bigr) da 
\enspace .
\end{equation}

Similarly, we assume no mortality at all between birth and butchering time for males and females in the butchery line, 
so the size of the cattle population suitable for butchery is (\emph{without any butchering before age $\Omega_1$})
\begin{equation}
\label{eq.butchery.potential}
N_b(t) = \int_{\Omega_0}^{\Omega_1} \! B_b(t-a) da 
= \int_{\Omega_0}^{\Omega_1} \! \rhoseason(t-a) m \bigl( N_r(t-a) \bigr) N_r(t-a) \bigl[ 2 - R  \bigl( P(t-a)  \bigr) \bigr] da 
\enspace . 
\end{equation}

\subsection{Market dynamics}
Inspired by \cite{Bel_Mac:1989} (see also \cite{Mac:1989}), we consider the following differential equation satisfied by the
price as a function of the demand $D(P)$ and the supply $S(t)$:
\begin{equation}
\label{eq.price}
\frac{P^{\prime}(t)}{P(t)} = \lambda F \Bigl( D \bigl( P(t) \bigr) , S(t) \Bigr) 
\quad \mbox{where} \quad F(D,S)=\frac{D-S}{D+S} \enspace .
\end{equation}
Other functions $F$ could be considered such as $F(D,S) = (D-S)/S$. The parameter $\lambda>0$ measures the ``temperature'' of the market, i.e., how fast can the price goes up or down.

The function $P \to D(P)$ is a decreasing function of the price $P$. 

\medskip

In order to define the supply function $S$, we assume that all the cattle in the butchery line is butchered exactly at age $\Omega_1$ (and never before), while the market takes into account all the population $N_b(t)$ for determining the price in equation \eqref{eq.price}.
This leads to choosing
\begin{equation}
\label{eq.modelB.supply}
S(t) 
= \frac{N_b(t)}{\Delta\Omega} 
= \frac{1}{\Delta\Omega} \int_{\Omega_0}^{\Omega_1} \rhoseason(t-a) 
m \bigl( N_r(t-a) \bigr) N_r(t-a) \bigl[ 2 - R \bigl( P(t-a) \bigr) \bigr]  da 
\enspace ,
\end{equation}
for $\Delta\Omega = \Omega_{1} - \Omega_{0}$. 
\subsection{The population/market model}

Our model then consists in coupling equations \eqref{eq.mature.females}, \eqref{eq.price} and \eqref{eq.modelB.supply}:
\begin{subequations}
\begin{align}
\label{eq.modelB.1}
N_r(t) &= \int_{A_0}^{A_1}   \rhoseason(t-a) m \bigl( N_r(t-a) \bigr) N_r(t-a) R \bigl( P(t-a) \bigr) da
\\
\label{eq.modelB.2}
\frac{P^{\prime}(t)}{P(t)} &= \lambda F \Bigl( D \bigl( P(t) \bigr),S(t) \Bigr)
\\
\label{eq.modelB.3}
S(t) &= \frac{1}{\Delta\Omega} \int_{\Omega_0}^{\Omega_1}  \rhoseason(t-a) 
m \bigl( N_r(t-a) \bigr) N_r(t-a) \bigl[ 2 - R \bigl( P(t-a) \bigr) \bigr] da  
\end{align}
\end{subequations}
where $F(D,S)=(D-S)/(D+S)$, and $m: [0,+\infty) \to [0,+\infty)$, $\rhoseason: \R \to [0,+\infty)$, 
$R: [0,+\infty) \to [0,1]$, $D: [0,+\infty) \to [0,+\infty)$, $A_1>A_0>0$, $ \Omega_{1}>\Omega_{0}>0$, $\Delta\Omega=\Omega_{1}-\Omega_{0}$ and $\lambda>0$ have to be chosen. 
Possible choices for all these parameters are discussed in Section~\ref{sec.values-param}.

\section{Analysis of the model}
\label{sec.math}
The section analyses mathematically the model defined by equations~\eqref{eq.modelB.1}--\eqref{eq.modelB.3}, under the following assumptions:
\begin{itemize}
\item $\rhoseason:\R\to\R$ is a non-negative, bounded, $1$-periodic function such that $\int_0^1\rhoseason=1$ and we
let $\rhoseason(t)\le\rhomax$ and
\begin{equation}\label{eq.emmero}
0<c_0\le\int_{A_0}^{A_1}\rhoseason(t-a)da\le c_1 \qquad \forall t;
\end{equation}
\item $m:\left[0,+\infty\right)\to\R$ is a continuous function that satisfies
\[
\frac{m_0}{2} \min\{1, N^{-\gamma}\} \le m(N) \le m_0 \min\{1, N^{-\gamma}\} \qquad \forall N > 0
\]
with $m_0>0$ and $\gamma\ge 1$;
\item $R:\left[0,+\infty\right)\to\R$ is a continuous function such that $R_0\le R(P)\le R_1$ for all $P \geq 0$ and some constants $R_1,R_0>0$
with $R_1 \leq 1$;
\item $D:\left[0,+\infty\right)\to\R$ is a strictly decreasing and locally Lipschitz continuous function such that $D(+\infty)=0$ and we set $D_{0}=D(0)$.
\end{itemize}
We begin by setting up a phase space and a notion of solution suitable for our model.
Let $T_0 = \min\{A_0, \Omega_0\}$, $T_1 = \max\{A_1, \Omega_1\}$ and
$\X = L^{\infty}([-T_1,0];\left[0, +\infty\right)) \times C^0([-T_1,0];\left[0, +\infty\right))$
which is a complete metric space with respect to the distance induced by the norm
\[
\bigl\lVert (N, P) \bigr\rVert_{\X} 
:= \lVert N \rVert_{\infty} + \lVert P \rVert_{\infty} 
:= \esssup_{s \in [-T_1, 0]} \bigl\lvert N(s) \bigr\rvert + \sup_{s \in [-T_1, 0]} \bigl\lvert P(s) \bigr\rvert 
\enspace .
\]
In particular, when we consider $ N \in L^{\infty}([-T_{1},0];\left[0, +\infty\right)) $,
we actually mean that $ \essinf N \ge 0 $.
\begin{definition}\label{def.IVP}
Let $(N_0, P_0) \in \X$, and $t_0, T \in \R$ with $ t_{0} < T $. A solution of \eqref{eq.modelB.1}--\eqref{eq.modelB.3} with initial data $(N_0, P_0)$ is a couple $(N_r,P): \left[t_0 - T_1, T\right) \to \R^2$ such that $N_r|_{\left[t_0,T\right)}$ is continuous, $P|_{\left[t_0,T\right)}$ is differentiable, $N_r,P$ satisfy \eqref{eq.modelB.1}--\eqref{eq.modelB.3} for $t \in \left[t_0, T\right)$, while $N_r(t_0+a)=N_0(a)$ and $P(t_0+a)=P_0(a)$ for $a\in\left[-T_1,0\right)$.
\end{definition}
Our first result shows that a unique solution exists, is globally defined and satisfies some estimates: $ N_{r}$ and $S$ are globally bounded
and the component $ N_{r} $ turns out to be Lipschitz continuous on $ \left[t_{0},+\infty\right) $.
In particular all the obtained estimates are uniform with respect to the initial condition.
\begin{proposition}\label{pro.IVP}
Let $(N_0, P_0) \in \X$ and $t_0 \in \R$ be given.
Then there exists a unique solution pair $(N_r,P):\left[-T_1+t_0,+\infty\right)\to\R^2$ of \eqref{eq.modelB.1}--\eqref{eq.modelB.3} with initial data $(N_0, P_0)$.
Moreover, $N_r, P$ are non-negative and
\[
\begin{aligned}
& N_r(t) \le N_{\max}           \quad \forall t\ge t_0 \\
& \bigl\lvert N_r(t)-N_r(s) \bigr\rvert \le L_1 |t-s| \quad \forall t,s\ge t_0 \\
& 0 \le S(t) \le S_{\max}       \quad \forall t\ge t_0,
\end{aligned}
\]
where:
\begin{equation}\label{eq.NmaxL1Smax}
\begin{aligned}
& N_{\max} := m_0 R_1 c_1, \qquad L_1 := 2 m_0 R_1 \rhomax \\
& S_{\max} := m_0 \frac{2-R_0}{\Delta\Omega} \sup_{s\in[0,1]} \int_{\Omega_0}^{\Omega_1} \rhoseason(s-a)da.
\end{aligned}
\end{equation}
\end{proposition}
\begin{proof}
We set $N_r(t)=N_0(t-t_0)$ and $P(t)=P_0(t-t_0)$ for $t\in\left[t_0-T_1,t_0\right)$ and remark that the equations \eqref{eq.modelB.1}--\eqref{eq.modelB.3} allow to extend $N_r$ and $S$ on $\left[t_0-T_1,t_0+T_0\right)$ in a unique and continuous way.
In particular we have that $N_r,S\ge 0$ and $N_r(t)\le m_0c_1R_1$ for $t\in\left[t_0,t_0+T_0\right)$ and $S(t)\le S_{\max}$ for $t\in\left[t_0,t_0+T_0\right)$ since $Nm(N)\le m_0$ for all $N$.
Then \eqref{eq.modelB.2} can be uniquely solved in $\left[t_0,t_0+T_0\right)$ with respect to $P$ with $S$ and $P(t_0)$ given.
Indeed, no blow-up can occur at or before $t_0+T_0$ since $F(D,S) \leq 1$ for all $D>0$ and $S \geq 0$.
An inductive argument shows that the same properties hold true on the interval $\left[t_0+(k-1)T_0,t_0+kT_0\right)$ for all $k\in\N$.

Now, let us fix any $s,t\ge t_0$ such that $s \le t$ and compute
\begin{align*}
\bigl\lvert N_r(t)-N_r(s) \bigr\rvert  
= & \left| \int_{t-A_1}^{t-A_0} N_r(\alpha) m\bigl( N_r(\alpha) \bigr) \rhoseason(\alpha) R \bigl( P(\alpha) \bigr) d\alpha \right.
\\
& \left. - \int_{s-A_1}^{s-A_0} N_r(\alpha) m\bigl( N_r(\alpha) \bigr) \rhoseason(\alpha) R \bigl( P(\alpha) \bigr) d\alpha \right|
\\
= & \left| \int_{s-A_0}^{t-A_0} N_r(\alpha) m\bigl( N_r(\alpha) \bigr) \rhoseason(\alpha) R \bigl( P(\alpha) \bigr) d\alpha \right.
\\
& \left. - \int_{s-A_1}^{t-A_1} N_r(\alpha) m\bigl( N_r(\alpha) \bigr) \rhoseason(\alpha) R \bigl( P(\alpha) \bigr) d\alpha \right|
\\
\le & m_0 R_1 \left( \int_{s-A_1}^{t-A_1} \rhoseason(\alpha) d\alpha + \int_{s-A_0}^{t-A_0} \rhoseason(\alpha) d\alpha \right)
\\
\le & 2m_0 R_1 \rhomax |t-s|. \qedhere
\end{align*}
\end{proof}
\begin{remark}\label{rem.discontinuity.at.t0}
In fact, equation \eqref{eq.modelB.1} prescribes the value $N_r(t_0)$ which may be different from $N_0(t_0^-)$.
Hence the solution component $N_r$ may have a jump discontinuity at $t_0$ even if $N_0$ is continuous.
However, $N_r$ is bounded and Lipschitz continuous on $\left[t_0,+\infty\right)$ with constants that do not depend on initial data.
On the other hand, it is clear from the proof of Proposition \ref{pro.IVP} that $P(t) > 0$ for all $t > t_0$ if and only if $P(t_0)>0$.
In particular, if $ P_{0} $ is not identically zero but satisfies $ P_{0}(t_{0}) = 0 $, then we have $P(t) = 0$ for all $t \ge t_0$,
no matter what is $N_0$.
\end{remark}
The next results show that, under suitable assumptions, all the components of the solution eventually are uniformly bounded
away from zero.
In particular, the conditions in statements $(2)$ and $(3)$ of Proposition~\ref{pro.uniform.persistence.for.N} 
require that the breeding strategy has to be suitably tuned to the maximal
fertility rate.
Moreover, the obtained estimates will allow to define a compact invariant set which absorbs in finite time all the relevant solutions.
\begin{proposition}\label{pro.uniform.persistence.for.N}
Let $(N_r, P)$ be the solution of \eqref{eq.modelB.1}--\eqref{eq.modelB.3} with initial data $(N_0, P_0) \in \X$ at time $t_0$.
\begin{enumerate}
\item
If $N_r(t) \le N_{\max}$ for a.a. $t \in [\hat{t} - A_1, \hat{t}]$ for some $\hat{t}  \geq t_0$, then
\[
N_r(t) \ge \frac{m_0R_0c_0}{2} \min\left\{ \essinf_{[\hat{t}-A_1,\hat{t}]} N_r, N_{\max}^{1-\gamma}\right\}
\quad \forall t \in \bigl[ \hat{t}, \hat{t}+A_0 \bigr] 
\enspace ,
\]
In particular this inequality holds for all $t \ge t_0 + A_1$ by Proposition~\ref{pro.IVP}.
\item
If $m_0R_0c_0>2$ and $\inf_{[\hat{t}-A_1,\hat{t}]} N_r \ge N_{\max}^{1-\gamma}$ for some $\hat{t} \ge t_0 + A_1$, then
\[
N_r(t) \ge N_{\min} \quad \text{and} \quad S(t) \ge S_{\min} \qquad \forall t\ge \hat{t},
\]
where:
\begin{equation}\label{eq.NminSmin}
 N_{\min} := \frac{m_0R_0c_0}{2} N_{\max}^{1-\gamma} \quad \text{and} \quad
S_{\min} := m_0 \frac{2-R_1}{2\Delta\Omega} N_{\max}^{1-\gamma} \inf_{s\in[0,1]} \int_{\Omega_0}^{\Omega_1}\rhoseason(s-a)da.
\end{equation}
\item
If $m_0R_0c_0>2$ and $ N_{0}(a) >0 $ for almost all $ a \in [-A_{1},0] $, then there exists $t^*\ge t_0$ such that
$N_r(t) \ge N_{\min}$ and $S(t) \ge S_{\min}$ for all $t\ge t^*$.
\end{enumerate}
\end{proposition}
\begin{proof}
(1) If $N_r(t) \le N_{\max}$ for a.a. $t \in [\hat{t} - A_1, \hat{t}]$ then
\begin{equation}\label{eq.NmN}
N_r(t) m(N_r(t)) \ge
\frac{m_0}{2} \min \left\{ \essinf_{[\hat{t}-T_1,\hat{t}]} N_r, N_{\max}^{1-\gamma} \right\}
\qquad \text{for a.a. } t \in [\hat{t} - A_1, \hat{t}]
\end{equation}
by our assumption on $m$,
and Statement 1 follows immediately from \eqref{eq.modelB.1}.

(2) If, moreover, $m_0 R_0 c_0 > 2$ and $\inf_{[\hat{t}-A_1,\hat{t}]} N_r \ge N_{\max}^{1-\gamma}$ for some $\hat{t} \ge t_0 + A_1$, then by Statement 1 we deduce that
\[
N_r(t) \ge \frac{m_0R_0c_0}{2} N_{\max}^{1-\gamma} = N_{\min} \ge N_{\max}^{1-\gamma} \qquad \forall t \in [\hat{t}, \hat{t}+A_0]
\]
and, by induction, we obtain that $N_r(t) \ge N_{\min}$ for all $t\ge \hat{t}$.
The inequality for $S$ just follows from \eqref{eq.modelB.3}, \eqref{eq.NmN} and the inequality just proved for $N_r$ and Statement 2 is proved.

(3)
First note that, even if $\essinf N_0 =0$, we have that $N_r(t) > 0$ for all $ t \in [t_{0},t_{0}+A_{0}] $ by \eqref{eq.emmero} and
\eqref{eq.modelB.1}.
An iteration of this argument shows that $N_r(t) > 0$ for all $ t \ge t_{0} $.

By Statement~2, if some $\widetilde{t} \geq t_0+A_1$ exists such that $\inf_{[\widetilde{t}-A_{1},\widetilde{t}]} N_r \geq N_{\max}^{1-\gamma}$,
then Statement~3 holds true with $t^* = \widetilde{t}$.
Let us assume this does not happen and show this implies a contradiction, which will end the proof of Statement~3.
In other words, we now assume that
\begin{equation}
\label{eq.pr.unif-persist-N.stat3}
\forall t \geq t_{0}+A_{1} \, ,  \qquad 0 < \inf_{[t-A_{1},t]} N_r <  N_{\max}^{ 1 - \gamma } \enspace .
\end{equation}
In particular, by Statement~1 with $\hat{t} = t_{0}+A_{1}$,
\[
\forall t \in [t_{0} + A_1, t_{0} + A_1 + A_0] \, , \quad N_r(t) \ge \frac{m_0R_0c_0}{2}\inf_{[t_{0}, t_{0} + A_1]} N_r >
\inf_{[t_{0}, t_{0} + A_1]} N_r
\]
and applying the same reasoning $k \geq 1$ times (since equation~\eqref{eq.pr.unif-persist-N.stat3} is assumed to hold for every $t \geq t_1$), we get that
\[
\forall t \in [t_{0} + A_1 + k A_0, t_{0} + A_1 + (k+1) A_0] \, , \quad 
N_r(t) \ge \frac{m_0R_0c_0}{2}\inf_{[t_{0}, t_{0} + A_1]} N_r 
\enspace . 
\]
Taking $k \geq A_1/A_0$ we get
\[
\inf_{[t_{0} + A_1, t_{0}+2 A_1]} N_r \geq \frac{m_0R_0c_0}{2}\inf_{[t_{0}, t_{0} + A_1]} N_r \enspace .
\]
Since equation~\eqref{eq.pr.unif-persist-N.stat3}  is assumed to hold for every $ t \geq t_{0}+A_{1} $, we can continue applying
similar estimates and show that for every $\ell \geq 1$,
\[
\inf_{[t_{0} + \ell A_1, t_{0}+ (\ell+1) A_1]} N_r \geq \paren{ \frac{m_0R_0c_0}{2} }^{\ell} \inf_{[t_{0}, t_{0} + A_1]} N_r \enspace ,
\]
which implies that the left-hand side tends to infinity as $\ell$ tends to infinity since $\frac{m_0R_0c_0}{2}>1$ and
$\inf_{[t_{0}, t_{0} + A_1]} N_r>0\,$.
This is in contradiction with the boundedness of $ N_{r} $ (see Proposition~\ref{pro.IVP}).
\end{proof}
\begin{proposition}\label{pro.price.dynamics}
Let $(N_r, P)$ be a solution of \eqref{eq.modelB.1}--\eqref{eq.modelB.3} and assume that some $t^* \geq t_0$ exists such that $0 < S_{\min} \le S(t) \le S_{\max}$ for all $t \ge t^*$, 
where we recall that $S_{\min}$ is defined in Proposition~\ref{pro.uniform.persistence.for.N} 
and $S_{\max}$ is defined in Proposition~\ref{pro.IVP}.
\begin{enumerate}
\item Let $P^* \geq 0$ be such that $D(P^*) < S_{\min}$. 
If $P(\hat{t}) > P^*$ for some $\hat{t} \ge t^*$, then we have
\[
P(t) < P^* \qquad \forall t > \hat{t} + \dfrac{P(\hat{t}) - P^*}{\lambda P^*} \cdot \dfrac{D_{0} + S_{\max}}{S_{\min} - D(P^*)}\,.
\]
\item Let $P_* \geq 0$ be such that $D(P_*) > S_{\max}$. 
If $0 < P(\hat{t}) < P_*$ for some $\hat{t} \ge t^*$, then we have
\[
P(t) > P_* \qquad \forall t > \hat{t} + \dfrac{P_* - P(\hat{t})}{\lambda P(\hat{t})} \cdot \dfrac{D_{0} + S_{\max}}{D(P_*) - S_{\max}}\,.
\]
\end{enumerate}
\end{proposition}
\begin{proof}
(1) As long as $P(t) \ge P^*$, $D(P) \leq D(P^*) < S_{\min}$ so that by equation~\eqref{eq.modelB.2} $P$ decreases,
\[
P^{\prime}(t) \le \lambda P^* \frac{D(P^*) - S_{\min}}{D_{0} + S_{\max}}
\]
and, thus,
\[
P(t) \le - \lambda P^* \frac{S_{\min} - D(P^*)}{D_{0} + S_{\max}} (t - \hat{t}) + P(\hat{t}) \enspace .
\]
Therefore, $P(t)$ reaches the level $P^*$ before the time
\[
\hat{t} + \dfrac{P(\hat{t}) - P^*}{\lambda P^*} \cdot \dfrac{D_{0} + S_{\max}}{S_{\min} - D(P^*)} \,,
\]
afterwards $P(t)$ remains below $P^*$ since $P(t) = P^*$ implies $P^{\prime}(t) \leq 0$ by equation~\eqref{eq.modelB.2}, which proves Statement~1.

(2) Statement~2 follows in a similar way once we observe that, as long as $P(t) \le P_*$, $S_{\max} > D(P_*) \geq D(P)$ so that by equation~\eqref{eq.modelB.2} $P$ increases and
\[
P^{\prime}(t) \ge \lambda P(\hat{t}) \frac{D(P_*) - S_{\max}}{D_{0} + S_{\max}} \enspace .
\qedhere
\]
\end{proof}
\begin{corollary}\label{cor.price.dynamics}
Assume that $m_0R_0c_0>2$, $D_{0}>S_{\max}$ and $S_{\min} > 0$ and let $C>1$ and $P_{\min}, P_{\max}$ be such that $[P_{\min}, P_{\max}] = D^{-1}([S_{\min}, S_{\max}])$.
\begin{enumerate}
\item If $ S(t) \ge S_{\min} $ for all $t \ge t^*$ and  $P(\hat{t}) \in [P_{\min} / C, C P_{\max}])$ for some $\hat{t} \ge t^*$, then $P(t) \in [P_{\min} / C, C P_{\max}])$ for all $t \ge \hat{t}$.
\item In any case, for every non trivial solution $(N_r, P)$ of \eqref{eq.modelB.1}--\eqref{eq.modelB.3} there exists $\hat{t}$ such that $P(t) \in [P_{\min} / C, C P_{\max}])$ for all $t \ge \hat{t}$.
\item Moreover, as long as $P(t)$ stays in $[P_{\min} / C, C P_{\max}])$, we have that $|P^{\prime}(t)| \le C \lambda P_{\max}$.
\end{enumerate}
\end{corollary}
\begin{proof}
Remark that the assumption $S_{\min}>0$, which is equivalent to requiring that
\[
R_1 < 2 \quad \mbox{and} \quad \inf_{t \in [0, 1]} \int_{\Omega_0}^{\Omega_1} \rhoseason(t-a) da > 0,
\]
ensures that $D^{-1}([S_{\min}, S_{\max}])$ is a compact interval.

(1) The first statement directly follows from equation~\eqref{eq.modelB.2} since $F(D(P),S) < 0$ for all $S \in [S_{\min},S_{\max}]$ and
$P > P_{\max}$, and $F(D(P),S) > 0$ for all $S \in [S_{\min},S_{\max}]$ and $P < P_{\min}$
(recall that $ S(t) \le S_{\max} $ for all $ t \ge t_{0} $ by Proposition~\ref{pro.IVP}).

(2) The second statement is a straightforward consequence of Proposition~\ref{pro.IVP}, Statement~3 of Proposition \ref{pro.uniform.persistence.for.N}, Proposition~\ref{pro.price.dynamics} with $(P_*,P^*)=(P_{\min}/C,C P_{\max})$ and the first statement.

(3) The third statement follows from equation~\eqref{eq.modelB.2} since $\absj{F(D,S)} \leq 1$ for all $D,S$.
\end{proof}
As is usual in delayed equations, we study the dynamical system produced on the space $ \X $ of initial conditions by the solutions of
\eqref{eq.modelB.1}--\eqref{eq.modelB.3}.
In particular, due to the natural periodicity of the seasonality function $ m_{\rho} $, we consider how the initial condition is
transformed after $1$ year.
Namely, let $\Pi : \X \to \X$ be defined by $\Pi(N_0, P_0) = (N_1, P_1)$ with $N_1(s) = N_r(1 + s)$ and $P_1(s) = P(1 + s)$ for $s \in [-T_1, 0]$,
where $(N_r, P)$ is the unique solution of \eqref{eq.modelB.1}--\eqref{eq.modelB.3} with initial data $(N_0, P_0)$ at $t_0=0$.
\begin{proposition}\label{pro.continuity}
$\Pi$ is continuous.
\end{proposition}
\begin{proof}
Let $(N_r, P), (M_r, Q)$ be the solution couples of \eqref{eq.modelB.1}--\eqref{eq.modelB.3} with initial data $(N_0, P_0), (M_0, Q_0) \in \X$, respectively, at time $t_0=0$ and let $S_{(N_0, P_0)}, S_{(M_0, Q_0)}$ be the corresponding supply functions given by \eqref{eq.modelB.3}.
Thanks to the continuity of $m$ and $R$ it is straightforward to show that $|N_r(t) - M_r(t)|$ and $|S_{(N_0, P_0)}(t) - S_{(M_0, Q_0)}(t)|$ can be made arbitrarily and uniformly small on $[0,T_0]$ provided that $\|(N_0 - M_0, P_0 - Q_0)\|_{\X}$ is small enough.
The same holds also for $|P(t) - Q(t)|$ on $[0, T_0]$ by standard results on the theory of ordinary differential equations (see \cite[Theorem II.3.2, p. 14]{Har:1964}).
It is sufficient to iterate this procedure a finite number $k$ of times such that $kT_0 \ge 1$ to complete the proof.
\end{proof}

Now we look for an attractor for $\Pi$.
The first step is to identify a compact invariant set that absorbs (almost) all orbits of the dynamical system.
Its definition is suggested by the propositions we proved and goes as follows.
\begin{definition} \label{def.K}
Let $\K$ be the set of couples $(N_0, P_0) \in \X$ such that they are Lipschitz continuous with constants $2 m_0 R_1 \rhomax$ and $2 \lambda P_{\max}$, respectively (see Proposition \ref{pro.IVP} and Corollary \ref{cor.price.dynamics}), they satisfy $N_{\min} \le N_0(s) \le N_{\max}$ and $P_{\min} / 2 \le P_0(s) \le 2 P_{\max}$ for almost all $s \in [-T_1, 0]$ and, moreover,
\begin{equation}\label{eq.Nat0}
N_0(0) = \int_{A_0}^{A_1} N_0(-a) m \bigl( N_0(-a) \bigr) \rhoseason(-a) D \bigl( P_0(-a) \bigr) da 
\enspace .
\end{equation}
\end{definition}
We cannot expect the basin of attraction of $\K$ to be the whole $\X$, since the equations \eqref{eq.modelB.1}--\eqref{eq.modelB.3} admit
semi-trivial solutions $(N_r, 0)$ and $(0, P)$ besides the trivial one.
Moreover, even if we provide initial data $(N_0, P_0) \in \X$ such that $ N_0 > 0$ and $P_0 \not \equiv 0$, we have that $P(t) = 0$
for all $t \ge 0$ if $P_0(0)=0$. 
Therefore we consider the following subset of $\X$:
\[
\X^* = \bigl\{ (N_0, P_0) \in \X: N_0(a)>0 \text{ a.e. in } [-A_{1}, 0] \text{ and } P_0(0) > 0 \bigr\} 
\enspace .
\]
We obtain the next result as an immediate consequence of Ascoli-Arzel\`a's Theorem, Propositions \ref{pro.IVP} and \ref{pro.uniform.persistence.for.N}, Corollary \ref{cor.price.dynamics} and the $1$-periodicity of $\rhoseason$.
\begin{proposition}\label{pro.invariant.compact}
Assume that $m_0 R_0 c_0 >2$, $D_{0} > S_{\max}$ and $S_{\min} > 0$.
Then $\K$ is a compact subset of $\X$, $\Pi(\K) \subseteq \K$ and for all $(N_0, P_0) \in \X^*$ there exists $k \in \N$ such that $\Pi^k(N_0,P_0) \in \K$.
\end{proposition}
Using the absorbing compact set $ \K $, we can now prove the existence of an attractor for all initial conditions of $ \X^{*} $.
\begin{theorem}\label{thm.global.attractor}
Assume that
\begin{align}
\label{hyp.thm.global.attractor.1}
&m_0 R_0 c_0 >2 \\
\label{hyp.thm.global.attractor.2}
&D_{0} > S_{\max} \\
\label{hyp.thm.global.attractor.3}
&S_{\min} > 0
\end{align}
with $ S_{\max} $ and $ S_{\min} $ given respectively in \eqref{eq.NmaxL1Smax} and \eqref{eq.NminSmin}, and let
\[
\Lambda = \bigcap_{k \in \N} \Pi^k(\K) \enspace .
\]
Then:
\begin{enumerate}
\item $\Lambda$ is a non-empty compact subset of $\X$;
\item $\Pi^k(\Lambda) = \Lambda$ for all $k \in \N$;
\item for each neighbourhood $U$ of $\Lambda$ and every $(N_0, P_0) \in \X^*$ there exists $k^* \in \N$ such that $\Pi^k(N_0, P_0) \in U$ for all $k \ge k^*$;
\item $\Lambda$ is an attractor of the dynamical system generated by the iterates of $\Pi$ and its basin of attraction contains $\X^*$.
\end{enumerate}
\end{theorem}
\begin{proof}
Statements~1 and~2 follow from the continuity of $\Pi$, the compactness of $\K$ and its $\Pi$-invariance.

Let $(N_0, P_0) \in \X^*$.
There exists $k_*$ such that $\Pi^k(N_0, P_0) \in \K$ for all $k \ge k_*$ by Proposition~\ref{pro.invariant.compact}.
Therefore, the sequence $\{\Pi^k(N_0, P_0)\}$ is relatively compact, all its limit points lie in $\Lambda$ and Statement 3 follows.

Let $V = (\Pi^{k_1})^{-1}(U)$ where $k_1 > T_1$ is a fixed integer and
\begin{align*}
U \egaldef \left\{ (N_0, P_0)  \in \X : \vphantom{\int}\right. & \essinf_{[-T_1, 0]}  N_0 > N_{\max}^{1-\gamma} \quad \text{ and } \\
                                                &\left. \frac{P_{\min}}{3} < P_0(s) < 3 P_{\max} \text{ for all } s \in [-T_1, 0] \right\}.
\end{align*}
$V$ is an open subset of $\X$, since $U$ is open and $\Pi$ is continuous, and contains $\K$ since $\Pi^{k_1}(\K) \subseteq \K \subset U$.
Thus, $V$ is an open neighborhood of $\Lambda$.
We claim that there exists an integer $k_2 \ge k_1$ such that $\Pi^{k_2}(V) \subseteq \K$.
Let us fix $(N_0, P_0) \in V$ and call $(N_r, P)$ the unique solution of \eqref{eq.modelB.1}--\eqref{eq.modelB.3} with initial data $(N_0, P_0)$ at time $t_0 = 0$.
Moreover we define $N_k(s) = N_r(k - s)$ and $P_k(s) = P(k - s)$ for $s \in [-T_1, 0]$, which means that $\Pi^k(N_0, P_0) = (N_k, P_k)$.
By construction $(N_{k_1}, P_{k_1}) \in U$ and, therefore $\inf_{[k_1-T_1, k_1]} N_r = \inf_{[-T_1, 0]} N_{k_1} > N_{\max}^{1 - \gamma}$ and, by Statement 2 of Proposition \ref{pro.uniform.persistence.for.N}, $N_r(t) \ge N_{\min}$ and $S(t) \ge S_{\min}$ for all $t \ge k_1$.
In particular $\inf_{[-T_1, 0]} N_{k} \ge N_{\min}$ for all $k \ge 2k_1$.
Next we observe that $P(t) \in (P_{\min} / 3,  3 P_{\max})$ for all $t \ge k_1$ by Corollary \ref{cor.price.dynamics} with $C = 3$.
The estimates of Proposition \ref{pro.price.dynamics} grants that, if $P(k_1) \not \in [P_{\min} / 2,  2 P_{\max}]$, then $P(t)$ enters (and thereafter remains in) the interval $[P_{\min} / 2,  2 P_{\max}]$ at a time which is bounded above by
\[
t_2 = k_1 + \max\left\{ \dfrac{D_{0} + S_{\max}}{2 \lambda [S_{\min} - D(2 P_{\max})]}, \dfrac{D_{0} + S_{\max}}{2 \lambda [D(P_{\min} / 2) - S_{\max}]} \right\},
\]
a number that is independent of $(N_0, P_0) \in V$.
The claim follows by the choice $k_2 \ge \max \{ 2 k_1, t_2 + T_1 \}$.

Therefore, we have that $\Pi^{k_2}(V) \subseteq \K \subset V$, $\bigcap_{k \in \N} \Pi^k(V) = \Lambda$ and $\Lambda$ is an attractor for the dynamical system generated by $\Pi$.
\end{proof}

Since the coefficient $\rhoseason$ is a $1$-periodic function, we argue now about the existence of $k$-periodic solutions of
\eqref{eq.modelB.1}--\eqref{eq.modelB.3}, i.e. fixed points of $\Pi^k$, $k \in \N$.
The analysis carried on till now allows to prove it as a straightforward application of Schauder fixed point theorem at least for $k \ge T_1$.
If we consider the set $\D \subset \X$ made of the couples $(N_0, P_0)$ such that they satisfy the same conditions
that define $\K$ with the only exception of \eqref{eq.Nat0}, we observe that $\D$ is a compact and convex subset of $\X$
(convexity is the reason why in the definition of $\D$ we have removed condition \eqref{eq.Nat0}).
Moreover, using the same notation introduced in the proof of Theorem~\ref{thm.global.attractor}, $N_k$ is a translation of the restriction of $N_r$ to the interval $[k - T_1, k]$ and, hence, it is still continuous since $ k - T_1 \ge 0 $ (see Remark \ref{rem.discontinuity.at.t0}).
Thus, $\D$ is $\Pi^k$-invariant by Propositions \ref{pro.IVP} and \ref{pro.uniform.persistence.for.N} and Corollary \ref{cor.price.dynamics} and Schauder fixed point theorem applies to $\Pi^k : \D \to \D$.

On the other hand, if $k < T_1$, this argument fails: indeed, $N_k$ may now have a jump discontinuity at time $-k$, since \eqref{eq.Nat0} is no more guaranteed, and $\D$ is no more $\Pi^k$-invariant.
However, the points of discontinuity may only appear at negative integer times and may only be of jump type, while between two successive integer times the component $N_k$ remains Lipschitz continuous with the same known constant.
This remark suggests in which way the definition of the set $\D$ should be modified in order to preserve its $\Pi^k$-invariance when $k < T_1$.
We detail the argument in the following result when $k = 1 < T_1$ and we remark that we did not pursue the question of finding the minimal period of $k$-periodic solutions: the fixed point of $\Pi^k$ obtained above might actually be a fixed point of $\Pi$.
\begin{theorem}\label{thm:1periodic}
Assume that \eqref{hyp.thm.global.attractor.1}--\eqref{hyp.thm.global.attractor.3} hold true.
Then equations \eqref{eq.modelB.1}--\eqref{eq.modelB.3} have a non-trivial $1$-periodic solution.
\end{theorem}
\begin{proof}
We will consider only the case $T_1 > 1$ since we already explained how to deal with the case $T_1 \le 1$.
Let $i_0 \ge 1$ be the largest integer such that $i_0 < T_1$ and let $J_i = (-i, -i + 1)$, for $i = 1, \dots, i_0$, while we let $J_{i_0 + 1} = (-T_1, -i_0)$.
Moreover, let $\D$ be made by the couples $(N_0, P_0) \in \X$ such that:
\begin{itemize}
\item $N_0$ is Lipschitz continuous on each open interval $J_i$, for $i = 1, \dots, i_0 + 1$, with constant $2 m_0 R_1 \rhomax$;
\item $P_0$ is Lipschitz continuous with constant $2 \lambda P_{\max}$;
\item $N_0(s) \in [N_{\min}, N_{\max}]$ and $P_0(s) \in [P_{\min}/2, 2P_{\max}]$ for all $s \in [-T_1, 0]$.
\end{itemize}
We note that the assumptions on $N_0$ ensure that the one sided limits $N_0(-i^+)$ and $N_0(-i^-)$ exist and are finite for all $i = 1, \dots, i_0$.

The set $\D$ is clearly convex and let us consider a sequence $\{(N_{0,n}, P_{0,n})\}_{n \in \N} \subset \D$ in order to prove that $\D$ is compact in $\X$.
By Ascoli-Arzel\`a's Theorem it is possible to extract a subsequence, that we still call $\{(N_{0,n}, P_{0,n})\}_{n \in \N}$, such that $P_{0,n}$ and $N_{0,n}|_{J_i}$ converge uniformly for $i=1, \dots, i_0 +1$.
It is easy to show that the convergence is with respect to the norm of $\X$ and that the limit belongs to $\D$, which, thus, is compact.

Now, by the very definition of $\Pi$, we have that, if $(N_r, P)$ is the solution of \eqref{eq.modelB.1}--\eqref{eq.modelB.3} with initial data $(N_0, P_0) \in \D$ at time $t_0 = 0$ and $(N_1, P_1) := \Pi(N_0, P_0)$, then $N_1(s) = N_0(s + 1)$ for $s \in \left[-T_1, -1\right)$, while $N_1(s) = N_r(s + 1)$ for $s \in [-1, 0]$.
Therefore, $N_1$ may have jump discontinuities only at $i = -i_0, \dots, -1$.
Propositions \ref{pro.IVP} and \ref{pro.uniform.persistence.for.N} and Corollary \ref{cor.price.dynamics} (with $C = 2$) grant that $\Pi(N_0, P_0) \in \D$ and Schauder fixed point theorem concludes the proof.
\end{proof}

\section{Results of simulation experiments} 
\label{sec.simus}

This section shows numerical experiments for the model described in Section~\ref{sec.our-model}.

\subsection{Parameters of the model}
\label{sec.values-param}

We discuss here in which way one could choose some of the model's parameters 
in order to take into consideration real-world cattle, mainly hog.
However, we recall that in our numerical experiments our choice will be to remain close to the values used in \cite{Arl:2004}
for \textit{Microtus Epiroticus} (see Subsection~\ref{sec.param.setting}), 
since our primary goal here is to check the effects that the
interaction with the market equation has on the dynamics of Yoccoz-Birkeland equation (and viceversa).

\subsubsection{Biological parameters}
\label{sec.values-param.biological}

For the common pork the following choices are supported from the literature  \cite{Sor:1979} 
as well as from discussion with farmers \cite{Zanda:2009}. 

In our experiments, similarly to \cite{Arl:2004}, 
we take the fertility function defined by 
\[ 
m(N) = m_0 \bigl( \max \{N,1\} \bigr)^{-\gamma} 
\enspace , 
\]   
where $\gamma$ and $m_0$ remain to be chosen. 
The parameter $ m_{0} $ should count the largest number of female pups 
(approximately one half of the total number of pups) a fertile female
can give birth to in a year.
For common pork it may reasonably range in the interval $ [10, 13] $ \cite{Kna_Hos:2013}.
On the other hand, when one considers races of pork, like Cinta Senese in Italy, 
which are bred in a semi-wild environment, a lower
choice of $ m_{0} $ should be chosen.
Assuming a maximum of two litters per year \cite{Cam:2009} and the data provided for instance in \cite{Fra_Pug:2007},
the value of $ m_{0} $ for Cinta Senese should range in the interval $ [5, 10] $.

The following table is taken from \cite{Sor:1979} and could be used to choose values for $ A_{0} $ depending on the species one is
considering:
\begin{center}
\begin{tabular}{|l|c|c|}
\hline
Species & age at puberty (months) & length of gestation (days) \\
\hline
sow     & $ [ 5, 8] $            & $ 113 $ \\
\hline
cow     & $ [7, 18] $           & $ 283 $ \\
\hline
sheep   & $ [6, 9] $            & $ 147 $ \\
\hline
mare    & $ [10, 24] $          & $ 336 $ \\
\hline
\end{tabular}
\end{center}

It is more difficult to suggest a choice for $ A_{1} $ since the maximal age of reproducing cattle is a quantity that depends heavily on the
breeding strategies of the breeder.
In our setting, $ A_{0} $ and $ A_{1} $ are taken as in \cite{Arl:2004}.
\subsubsection{Seasonality}
We assume that the births are synchronized \cite{Zanda:2009} 
and we consider the following $1$-periodic piecewise-constant seasonal factor:
\[
\rhoseason(t) = \frac{1}{1-\rhowinter} \un_{t - \lfloor t \rfloor \in [0,1-\rhowinter)} 
\enspace .
\]
\subsubsection{Demand function}
In the numerical experiments we perform in this section, we choose the following 
demand function 
\begin{equation} \label{eq.demand-exp} D(P) = D_{\exp}(P) \egaldef D_0 e^{ -\alpha_D P } \end{equation}
for some parameters $D_0,\alpha_D>0$.
For simplicity we took the parameters $ \Omega_{0} $ and $ \Omega_{1} $ in the supply function to be equal to $ A_{0} $ and $ A_{1} $,
respectively.
As we have already noted the normalization of the supply is $\Delta\Omega=\Omega_1-\Omega_0$. 
\subsubsection{Breeder strategy}
The function $R$ defines how the breeder takes into account the current meat price 
for deciding how to split newborn females among the reproduction line (for a long-term increase of the supply) 
and the butchery line (for a short-term increase of the supply, but still delayed). 
A short term strategy would be to take $R$ close to 0 when the price is high. 
A long term strategy is to take $R$ close to 1 when the price is high. 
In the numerical experiments we choose 
\begin{equation} \label{eq.breeder-logistic} R(P) = R_{\mathrm{logistic}}(P) \egaldef R_0 + (R_1-R_0) f_d(P/P_0) \end{equation}
where 
\[
f_d(x) = \begin{cases}
\frac{x^d}{2}                        &\qquad \mbox{if } x \in [0,1) \\
\frac{1}{1 + \exp\paren{-2 d (x-1)}} &\qquad \mbox{otherwise} \end{cases}
\]
and $R_0, R_1 \in [0,1]$ and $P_0, d > 0$ have to be chosen. 

\subsubsection{Parameter setting}\label{sec.param.setting}
The default set of parameters (called \HH, because it is close to realistic values for pork, 
see Section~\ref{sec.values-param.biological}) and functions is the following: 
\begin{itemize}
\item population dynamics: 
$A_0=0.18$, 
$A_1=2$, 
$m$ defined by \eqref{eq.density.dependent.fertility} with 
$m_0=5$ and 
$\gamma=8.25$, 
and the 
seasonality function is 
\[ \rhoseason(t) = \frac{1}{1-\rhowinter} \un_{t - \lfloor t \rfloor \in [0,1-\rhowinter)} \]  with $\rhowinter=0.79$. 
\item market dynamics: $\lambda = 1$ and the demand function is $D=D_{\exp}$ is defined by equation~\eqref{eq.demand-exp}  
with $D_0 = 5$ and $\alpha_D =1$.
\item interaction between population and market: 
$\Omega_0=0.18$, 
$\Omega_1=2$, 
and 
$R=R_{\mathrm{logistic}}$ is defined by equation~\eqref{eq.breeder-logistic} 
with $R_0 = 0$ (minimal value), $R_1=1$ (maximal value), $P_0=1$ (price threshold) and $d=4$ (``degree of $R(P)$ for small $P$''). 
\end{itemize}

The setting \HH\ was chosen close to the main setting studied in \cite{Arl:2004} (up to the changes in the model),
with $m_0$ one order of magnitude smaller 
(i.e., 5 instead of 50, which is more realistic, and necessary to obtain a reasonable behaviour
because of the change in the survival function) and $\rhowinter$ modified 
(the latter choice resulted from a rough exploration of the
main parameters of the model). 
We remark that the choice of parameters in \HH\ is not far from realistic ones for livestock production. 

In our numerical experiments, only a couple of parameters will be varying: $m_0$ and $\gamma$, 
which are the ones that mostly influence the population dynamics.
In particular, figures \ref{fig.bifurc.HH-gamma.1d.Nr} and \ref{fig.bifurc.HH-gamma.1d.P} 
show that the model has a chaotic behavior also for
values of $ \gamma $ close to $2$, i.e. a much weaker dependence of fertility on over-population.

\subsection{Study of one parameter set (setting \HH)}  \label{sec.results.HH}

In this section, we focus on the setting \HH, for which some interesting behaviour can be observed. 

\subsubsection{Continuous time dynamics} \label{sec.results.HH.cont}
\begin{figure}
\begin{center}
\includegraphics[width=0.9\textwidth]{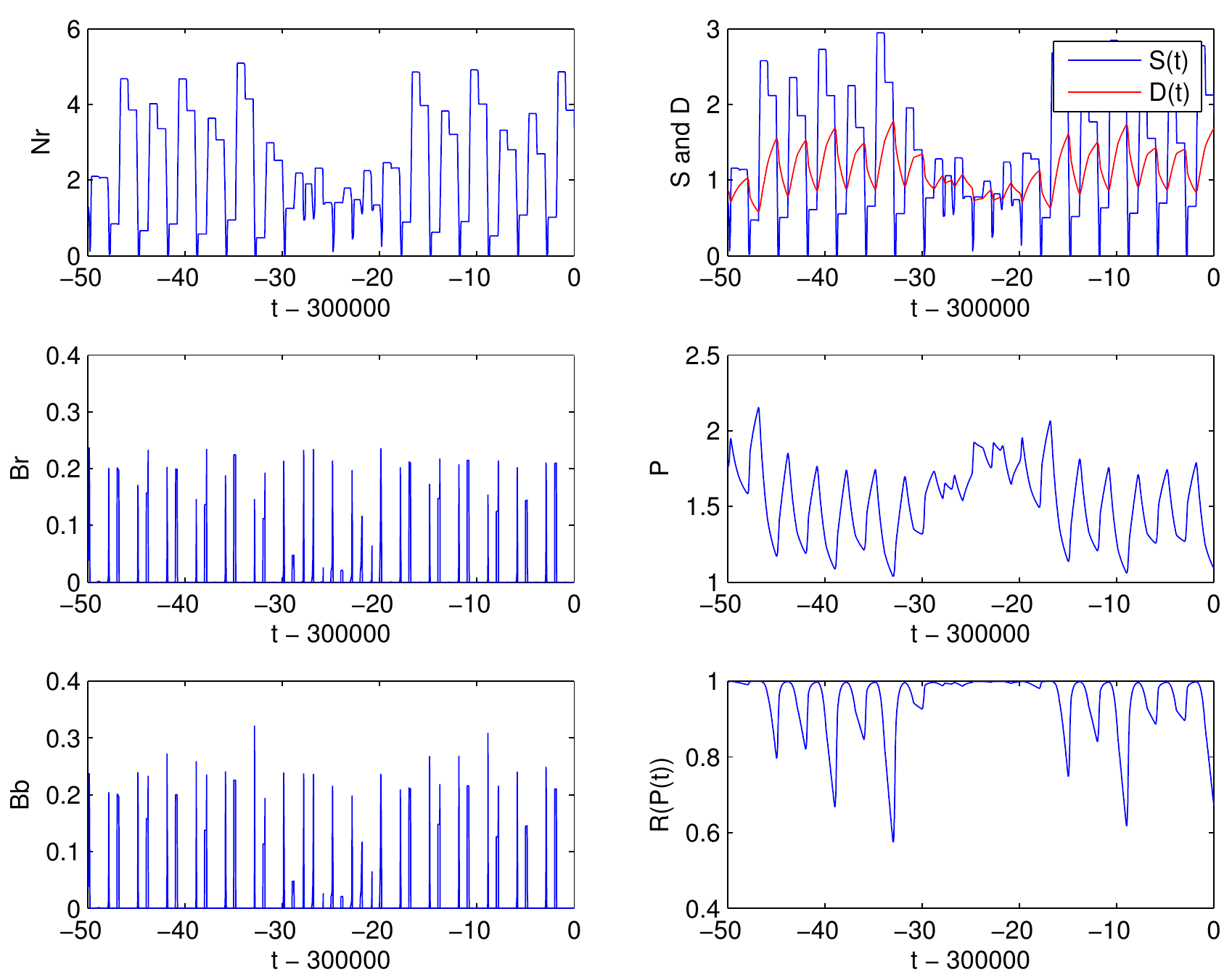}
\end{center}
\caption{\label{fig.HH.cont-50-6graphs}Continuous time dynamics for setting \HH\ over 50 years: $N_r$ (top left), $S$ and $D$ (top right), $B_r$ (middle left), $P$ (middle right), $B_b$ (bottom left) and $R(P)$ (bottom right). }
\end{figure}
The continuous time dynamics of setting \HH\ can be visualized on Figure~\ref{fig.HH.cont-50-6graphs} over 50 years.
Note that on Figure~\ref{fig.HH.cont-50-6graphs}
the mature reproducing population $N_r(t)$ goes through very small values, of order $10^{-3}$, which might seem unrealistic. 
This phenomenon can be interpreted as follows. 
There is no mortality, except for animals reaching the maximal ages $A_1$ and $\Omega_1$. 
So, if most of the mature reproducing population $N_r$ was born during a short period of time, $N_r(t)$ will drop down to (almost) zero $A_1$ years later. 
Here, we observe such drops, which do not endanger the whole population because it holds during the reproducing season and it doesn't occur too fast (the reproducing females were not all born during a too short time period). 
Therefore, once the drop has started, as soon as $N_r$ goes below 1, a large number of birth will happen, leading to mature reproducing females after a time delay $A_0$. 
So, even if $N_r$ was going exactly to zero during this delay period, it would increase again as soon as the newborn females become mature. 

A good way to visualize this phenomenon is given by Figure~\ref{fig.HH.cont-short-check-Nr} where $N_r(t)$ is plotted (left) together with the totale female population in the reproducing line (right): the latter quantity never goes below $0.4$.

\begin{figure}
\begin{center}
\includegraphics[width=0.45\textwidth]{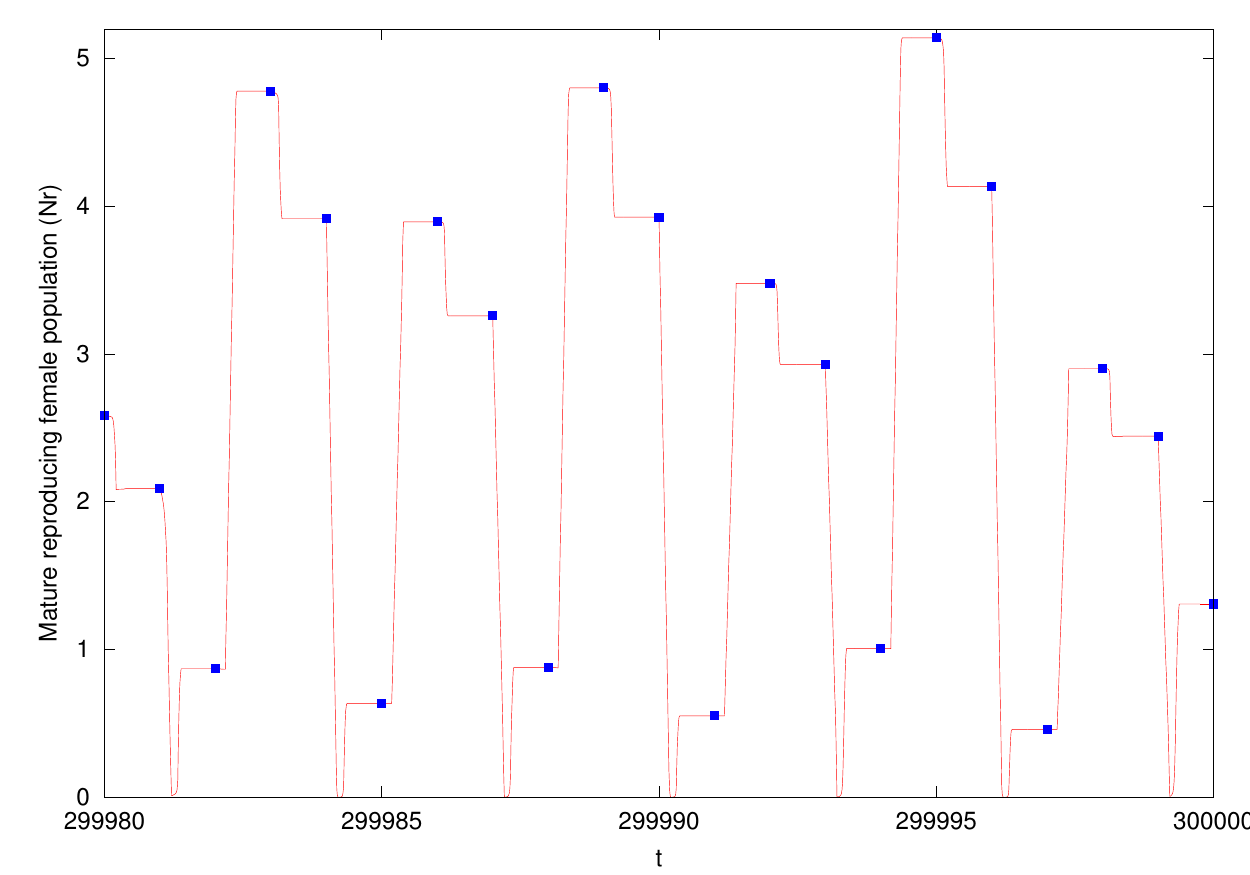}
\includegraphics[width=0.45\textwidth]{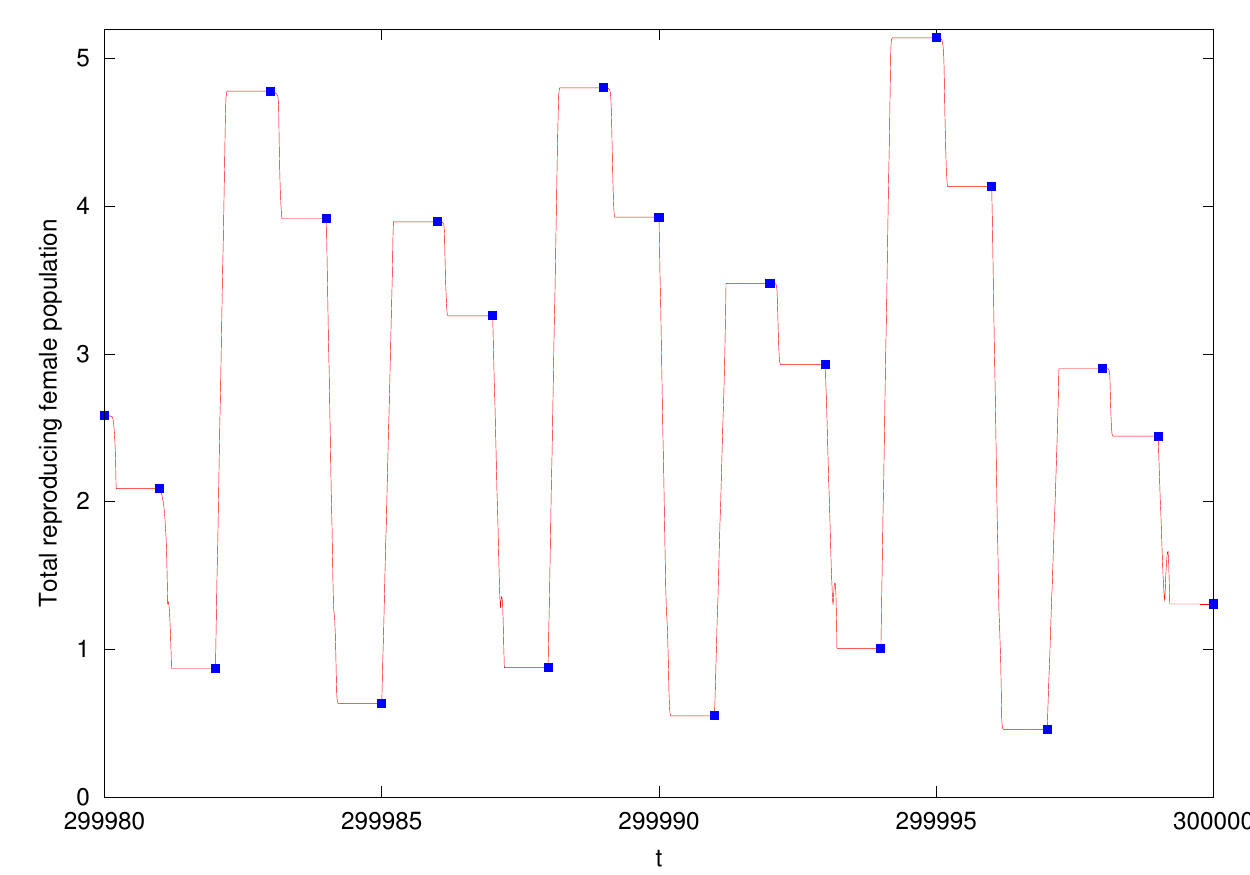}
\end{center}
\caption{\label{fig.HH.cont-short-check-Nr} 
Setting \HH, continuous time dynamics of reproducing female population over 20 years: 
mature ($N_r$, left) and total population (right).}
\end{figure}

A similar phenomenon holds with the butchery line (see Figure~\ref{fig.HH.cont-short-check-S}), 
where the ``mature'' butchery population (proportional to the supply $S(t)$) 
goes close to zero (left) but not the total butchery population (right). 

\begin{figure}
\begin{center}
\includegraphics[width=0.45\textwidth]{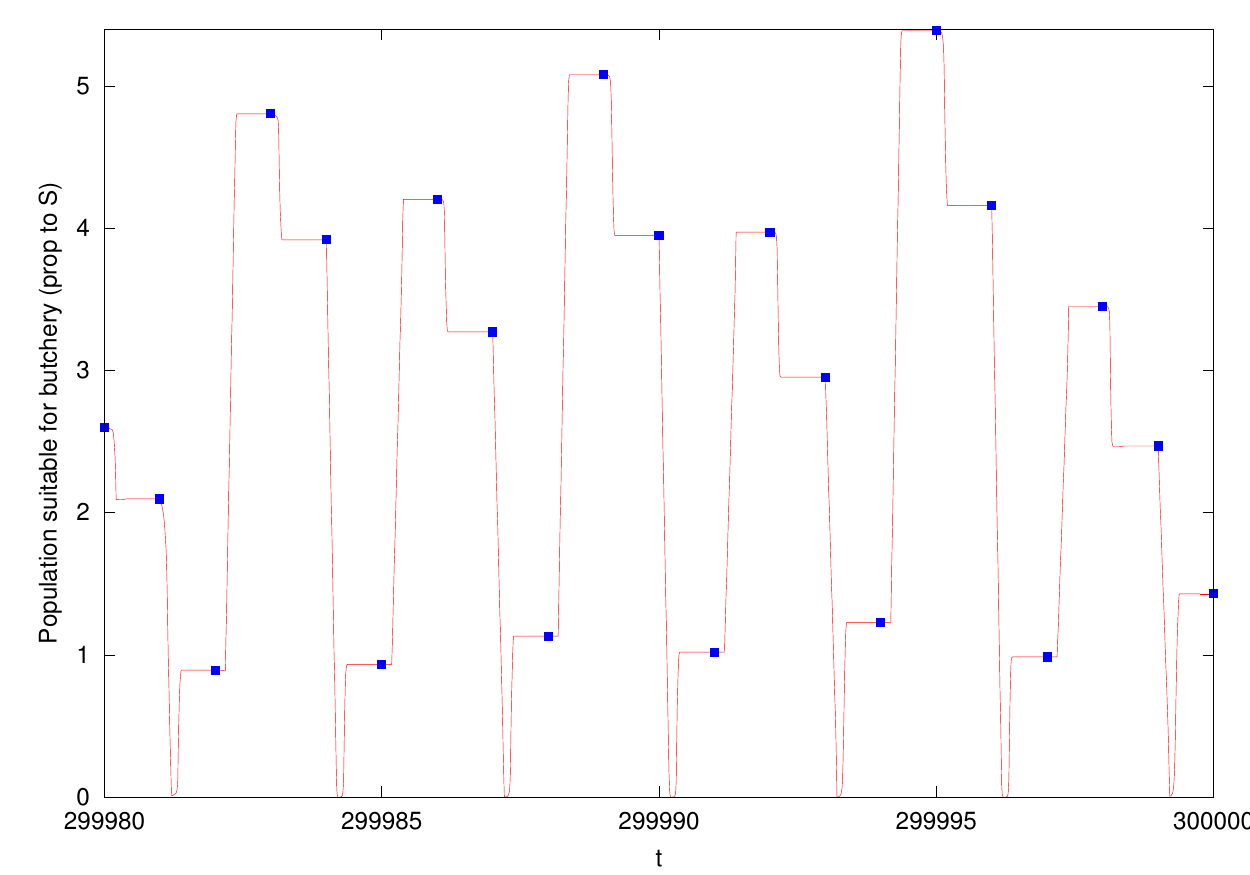}
\includegraphics[width=0.45\textwidth]{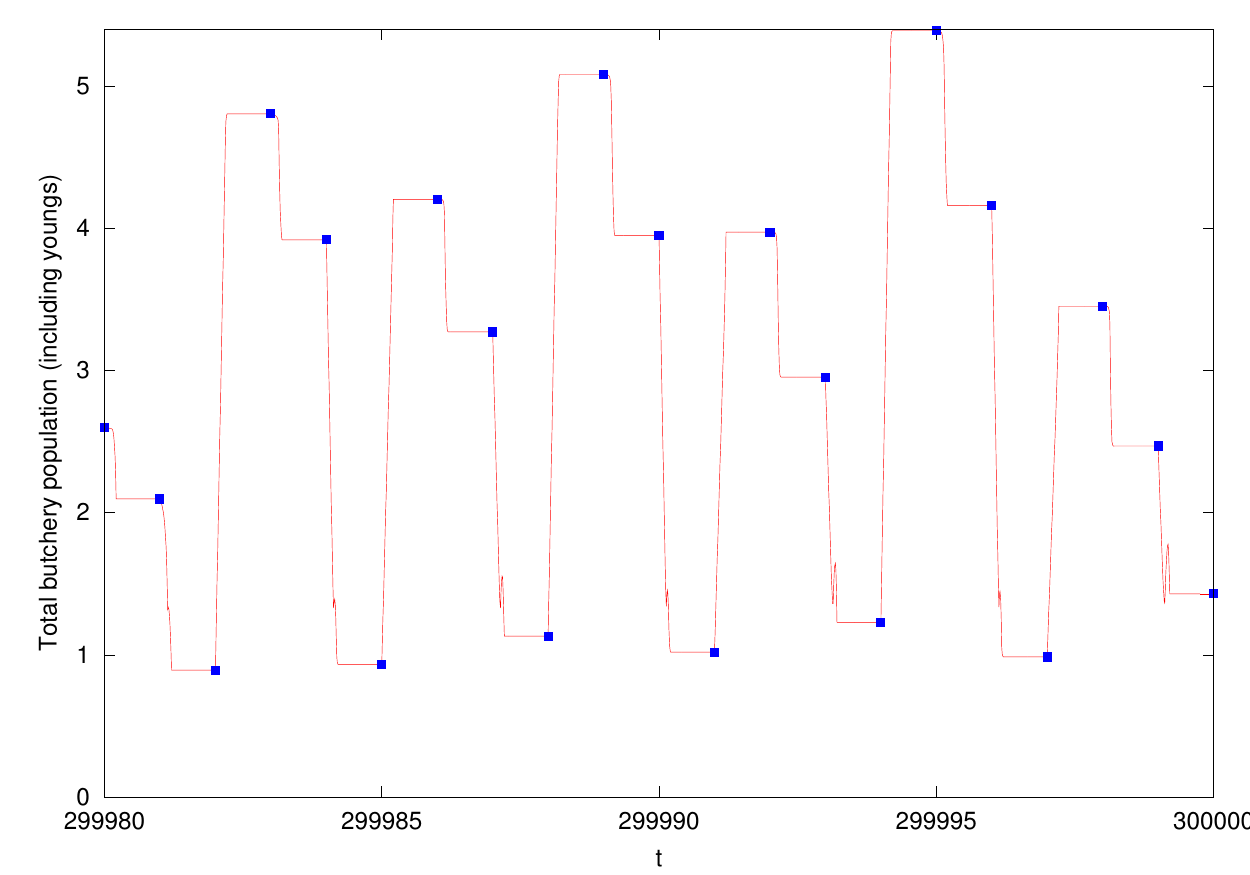}
\end{center}
\caption{\label{fig.HH.cont-short-check-S} 
Setting \HH, continuous time dynamics of butchery population over 20 years: 
``mature'' ($\propto S$, left) and total population (right).}
\end{figure}

\subsubsection{Yearly dynamics} \label{sec.results.HH.3d}
A natural way to display the behavior of the dynamical system we are studying is to only look at 
the values it takes for $t \in \N$, that is, at the very beginning of the birth period.
We recall that the continuous time dynamics takes place in an infinite-dimensional space $ \X $ corresponding to couples of
functions $ ( N(t), P(t) ) $, where $ t $ varies in some interval.
Our discretization (see Appendix~\ref{app.technical-details}) leads to a discrete dynamical system in a phase space of dimension 
$2 \times 201 $.
Then, we can visualize the dynamics by plotting in $\R^3$ the set
$\set{ (N_r(t), N_r(t+1), N_r(t+2) ) \, , \, t \in \N}$, as shown by Figure~\ref{fig.HH.Nr-3d.less-pts},
or the set $\set{ (P(t), P(t+1), P(t+2) ) \, , \, t \in \N}$, as shown by Figure~\ref{fig.HH.P-3d.less-pts}.
Both figures are the projection of the same attractor on two different subspaces.
We estimate the fractal dimension of the first set to $1.52$ and of the second one to $1.84$.
It seems reasonable to conjecture the existence of a strange attractor of dimension $ d \in (1,2) $.
See Section~\ref{app.technical-details.dim_f},  Figure~\ref{fig.HH.Nr-dim_f} and Figure~\ref{fig.HH.P-dim_f},
for details on how we estimate the fractal dimension. 
Note that Figure~\ref{fig.HH.Nr-3d.less-pts} shows an attractor similar to the one of the Yoccoz-Birkeland model with
$\rhowinter = 0.30$, $\gamma = 8.25$, $A_0 = 0.18$ and $ m_{0}=50 $ \cite[Figure~12]{Arl:2004}; 
only the center of the attractor here seems more complex. 
On the other hand, the interaction between the population and the price equation is crucial, 
as shown by Section~\ref{sec.results.HH.JA} below. 
\begin{figure}
\begin{center}
\includegraphics[width=0.8\textwidth]{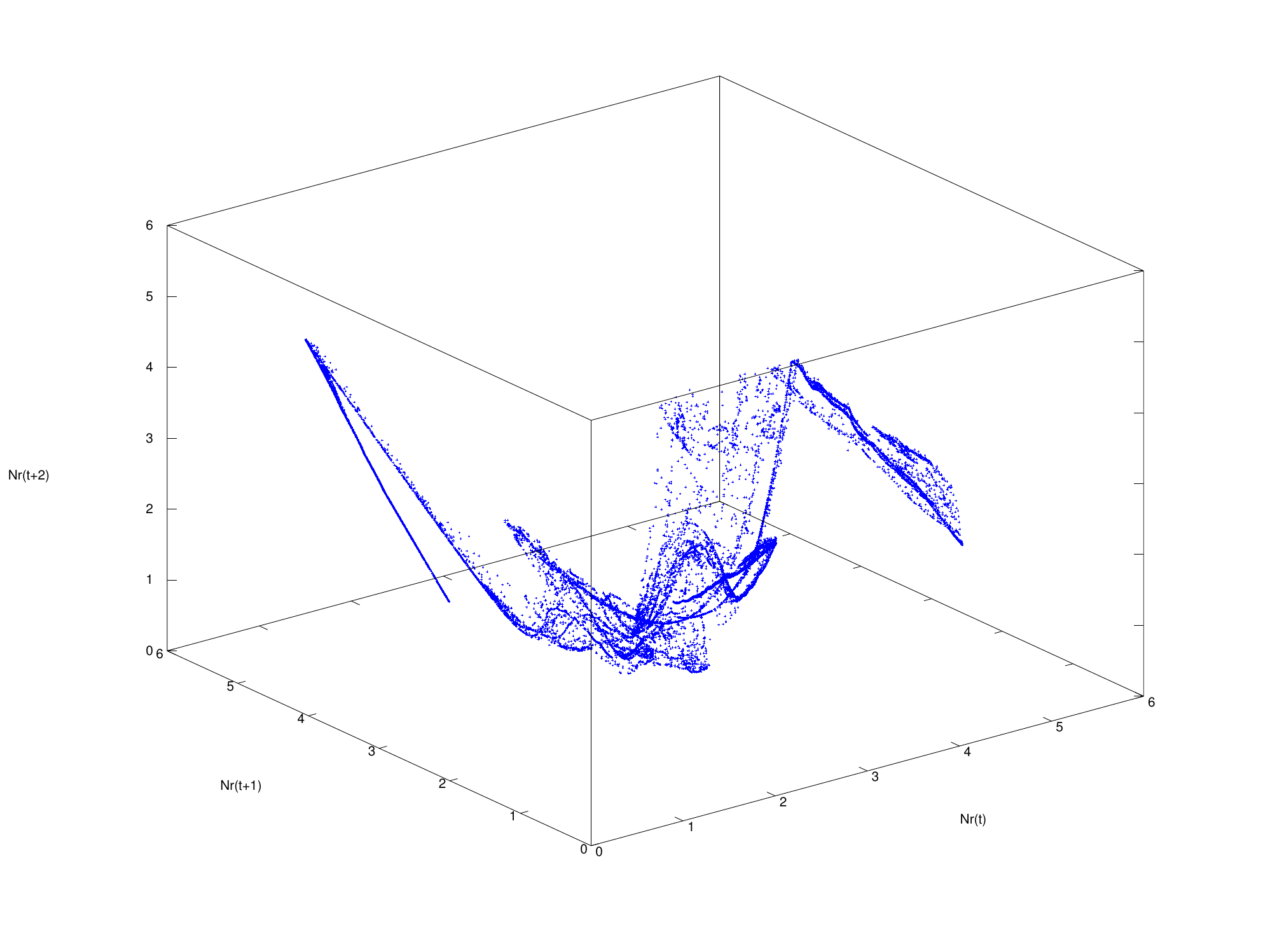}
\end{center}
\caption{\label{fig.HH.Nr-3d.less-pts} 
3d plot of $( N_r(t), N_r(t+1), N_r(t+2) ) $ with $ t \in \N $, $ 280\,000 \leq t \leq 300\,000 $ for setting \HH. 
}
\end{figure}
\begin{figure}
\begin{center}
\includegraphics[width=0.8\textwidth]{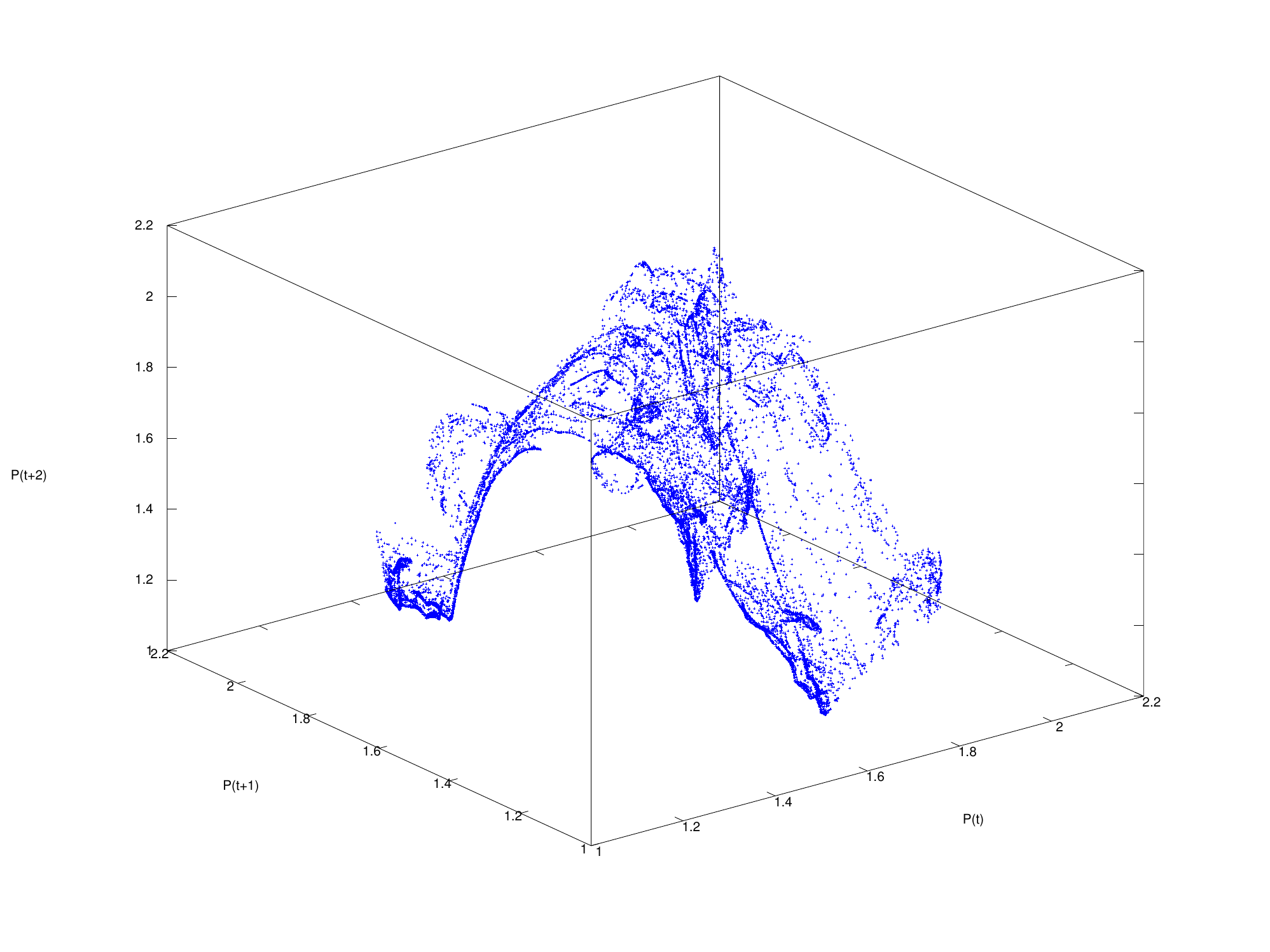}
\end{center}
\caption{\label{fig.HH.P-3d.less-pts} 
3d plot of $ ( P(t), P(t+1), P(t+2) ) $, with $ t \in \N$, $ 280\,000 \leq t \leq 300\,000 $ for setting \HH. 
}
\end{figure}

\subsubsection{Comparison with setting \JA: \HH\ with $R$ constant} \label{sec.results.HH.JA}
Let us call \JA\ the setting \HH\ with the value of the fraction of reproducing females in \eqref{eq.modelB.1} 
freezed: $R(P) = R_{\mathrm{cst}} \approx 0.955$ 
for all $P>0$. 
The constant value taken for $R$ is equal to the empirical average of $R(P(t))$ in setting \HH\ for $t \in [290\,000 , 300\,000]$. 
So, comparing results obtained with \JA\ and \HH\ shows the effect of the coupling between price and population. 
Note also that setting \JA\ then is an instance of the Yoccoz-Birkeland model with $A_0 = 0.18$, $A_1=2$, $\gamma=8.25$ and $m_0 \approx 4.78$ 
(with slightly different functions $m$ and $\rhoseason$, compared to \cite{Arl:2004}). 

The yearly dynamics of setting \JA\ can be visualized on Figure~\ref{fig.JA.3d} for $95\,000 \leq t \leq 100\,000$: 
it exihibits only a low-complexity seemingly nonchaotic orbit, very close to being periodic and totally 
different also from the original Yoccoz-Birkeland attractor.  
\begin{figure}
\begin{center}
\includegraphics[width=0.8\textwidth]{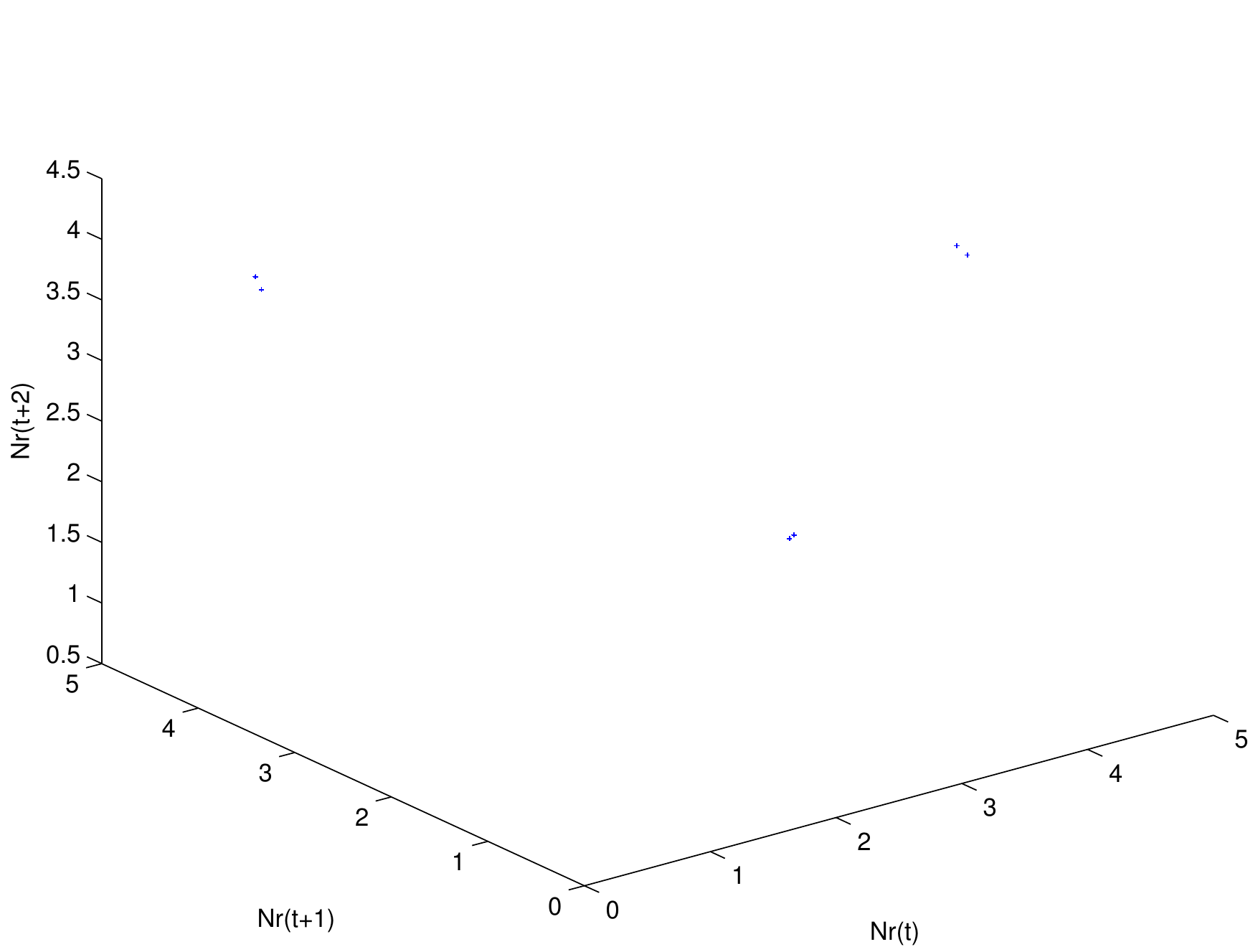}
\end{center}
\caption{\label{fig.JA.3d}3d plot of $ ( N_r(t), N_r(t+1), N_r(t+2) ) $, with $ t \in \N $, $ 95\,000 \leq t \leq 100\,000 $ for setting \JA. 
}
\end{figure}

\subsubsection{Analysis of chaos} \label{sec.results.HH.chaos}
In order to analyze the chaoticity of the dynamics of the price $P$, we followed a time series approach as in \cite{Kan_Sch:2003}: we computed
the autocorrelation function and determined its first zero $ \tau^{*} $.
Then we sampled the price time series at time steps multiple of $ \tau^{*} $ and we computed the combinatorial entropy for the binary sequence
obtained by looking at the sign of price returns (see Figure~\ref{fig.HH.P-chaos}). 

Given a times series $(Y(t))_{t \geq 0}$, its autocorrelation function is defined by  
\[ 
\forall \tau>0 \, , \quad 
R_Y(\tau) 
\egaldef \frac{ \moyB{ \parenb{ Y - \moy{Y} }  \parenB{ Y(\cdot+\tau) - \moyb{Y(\cdot+\tau)} } } } { \sqrt{  \moyB{ \parenb{ Y - \moy{Y} }^2 }  \moy{ \parenB{ Y(\cdot+\tau) - \moyb{Y(\cdot+\tau)} }^2 } } } 
\enspace,
\]
where $\moy{\cdot}$ means an average over time. 
In other words, the autocorrelation $R_Y(\tau)$ is the correlation between $Y(t)$ and $Y(t+\tau)$ 
for a random $t$ chosen uniformly in $[0,T]$, for some large time window $T>0$. 
So, roughly speaking, $R_Y(\tau)$ close to 0 means that $Y(t)$ does not provide information for predicting $Y(t+\tau)$. 

The autocorrelation function of $(P(t))_{t \geq 0}$ is plotted on the top left of Figure~\ref{fig.HH.P-chaos}. 
Its absolute value tends to be smaller for larger values of $\tau$, and it crosses zero several times, 
first for $\tau = \tau^{\star} \approx 1.37$.
Therefore, the discrete dynamical system $(P(k \tau^{\star}))_{k \in \N}$ 
is a good candidate for being unpredictable. 
So, we consider the corresponding ``returns'' 
\[ 
r_k \egaldef \log_{10} P \bigl( (k+1) \tau^{\star} \bigr)  - \log_{10} P(k \tau^{\star})
\]
(which are plotted on the top middle graph of Figure~\ref{fig.HH.P-chaos}), 
and evaluate the combinatorial entropy $H_K$ of the binary sequence 
$( (\sign(r_i))_{k-K+1 \leq i \leq k} )_{k \in \N}$ for various values of $K$
\footnote{The procedure we follow provides a lower bound to the Kolmogorov-Sinai entropy of the flow.
Indeed, this is defined as the supremum over all finite partitions of the rate of change of entropy
due to the finer partitioning given by the flow at each time step.
Here we have just fixed a particular type of partition (corresponding to the choice of sign of the returns).
The use of $\tau^*$ instead of $1$ as a time unit only affects entropy by a multiplicative factor.}.

The top right plot of Figure~\ref{fig.HH.P-chaos} shows $H_K$ as a function of $K$: 
there is a clear linear trend, with a positive slope $0.611$ (correlation coefficient $0.995$), 
which suggests a positive entropy. 

We also tried the same analysis for the discrete dynamical system $(P(k))_{k \in \N}$, 
that is, the time series at integer times, 
for which the autocorrelation function is not estimated to be (almost) zero. 
The results are given on the bottom part of Figure~\ref{fig.HH.P-chaos}, 
with the corresponding returns (middle) and entropy $H_K$ as a function of $K$ (right): 
linear regression still suggests a linear trend with a positive slope $0.647$ 
(correlation coefficient $0.983$), even if the returns seem less unpredictable 
as with $\tau^{\star}$ time steps. 
\begin{figure}
\begin{center}
\includegraphics[width=0.3\textwidth]{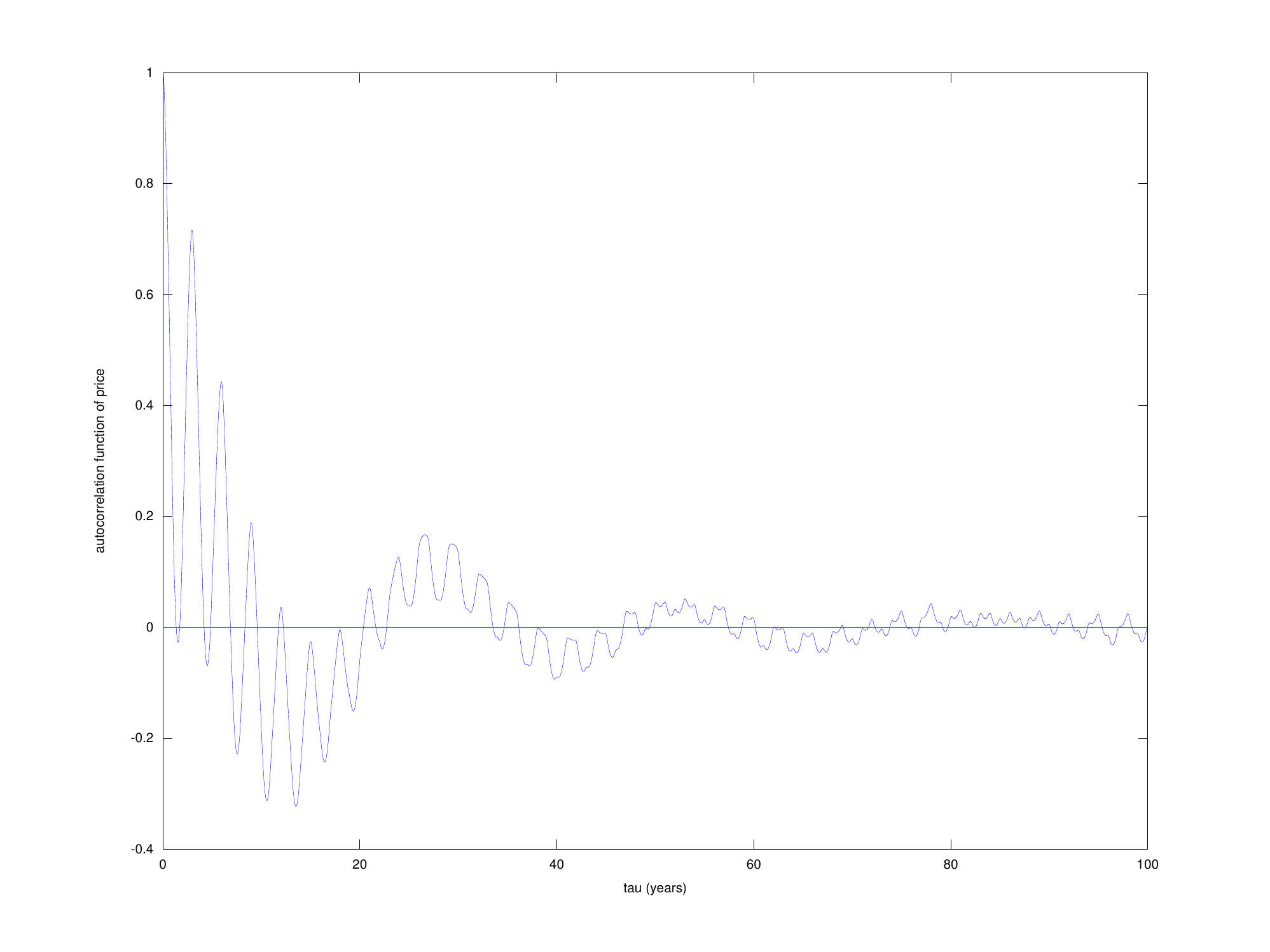}
\includegraphics[width=0.3\textwidth]{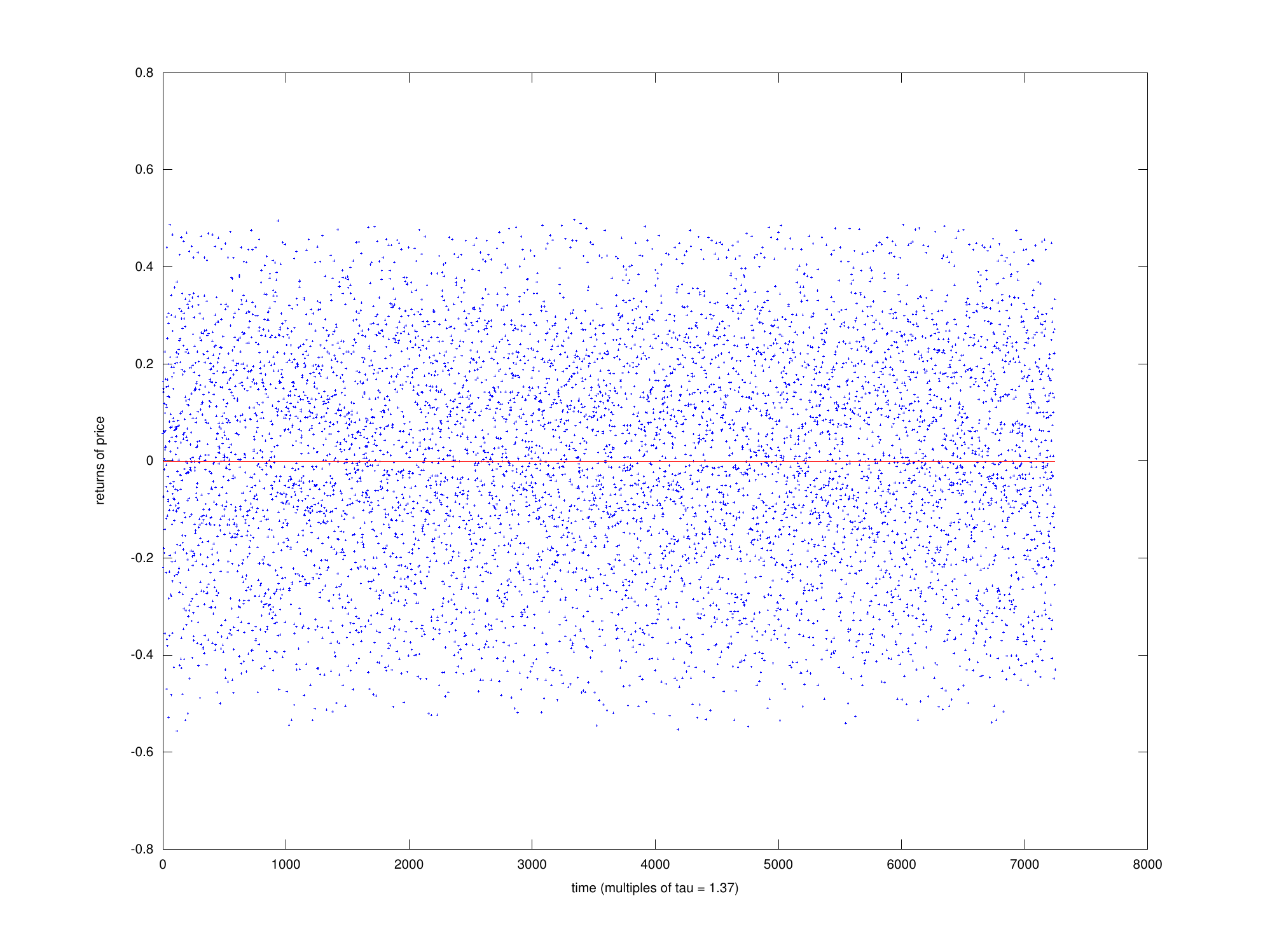}
\includegraphics[width=0.3\textwidth]{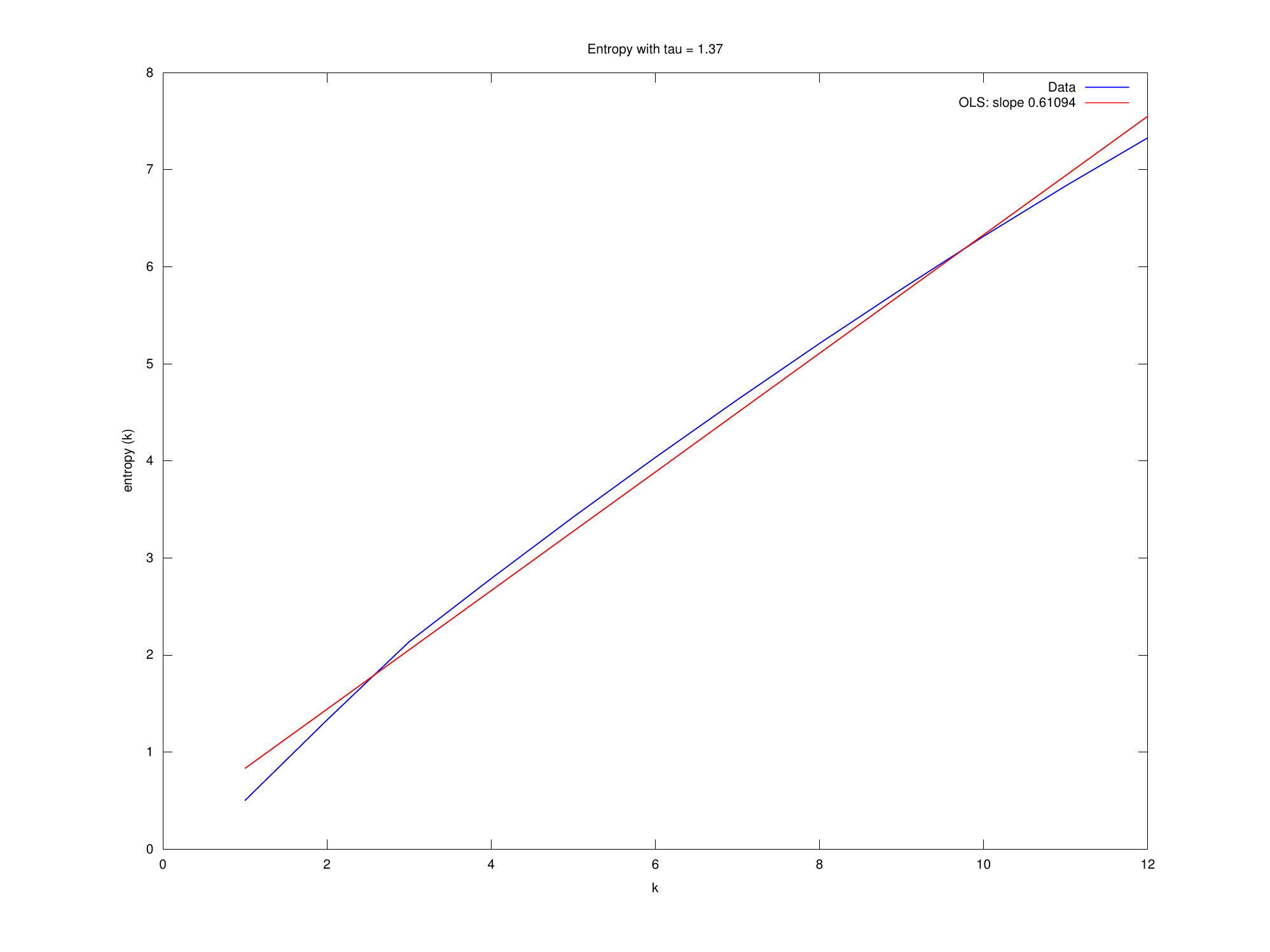}
\\
\hspace{0.3\textwidth}
\includegraphics[width=0.3\textwidth]{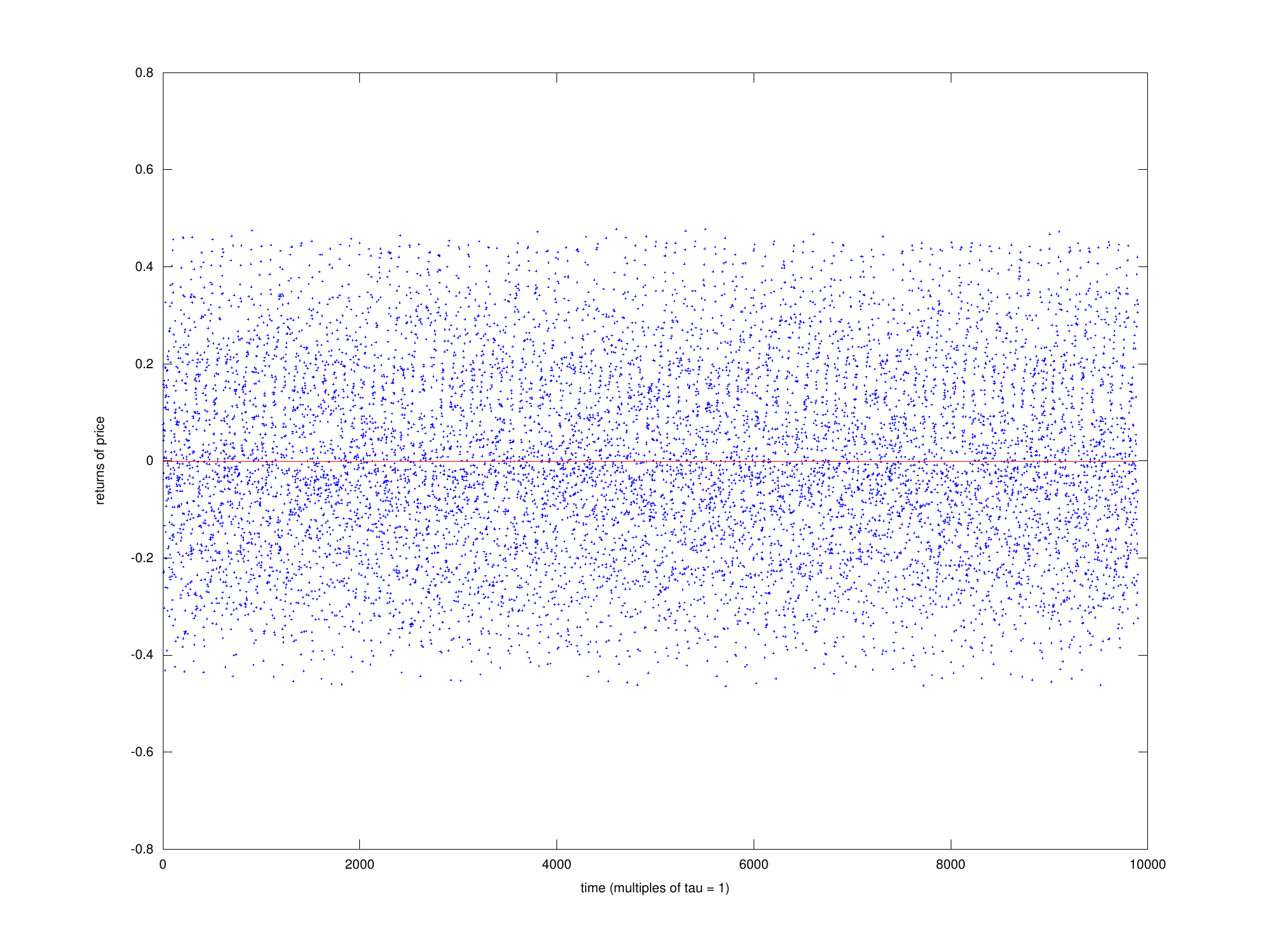}
\includegraphics[width=0.3\textwidth]{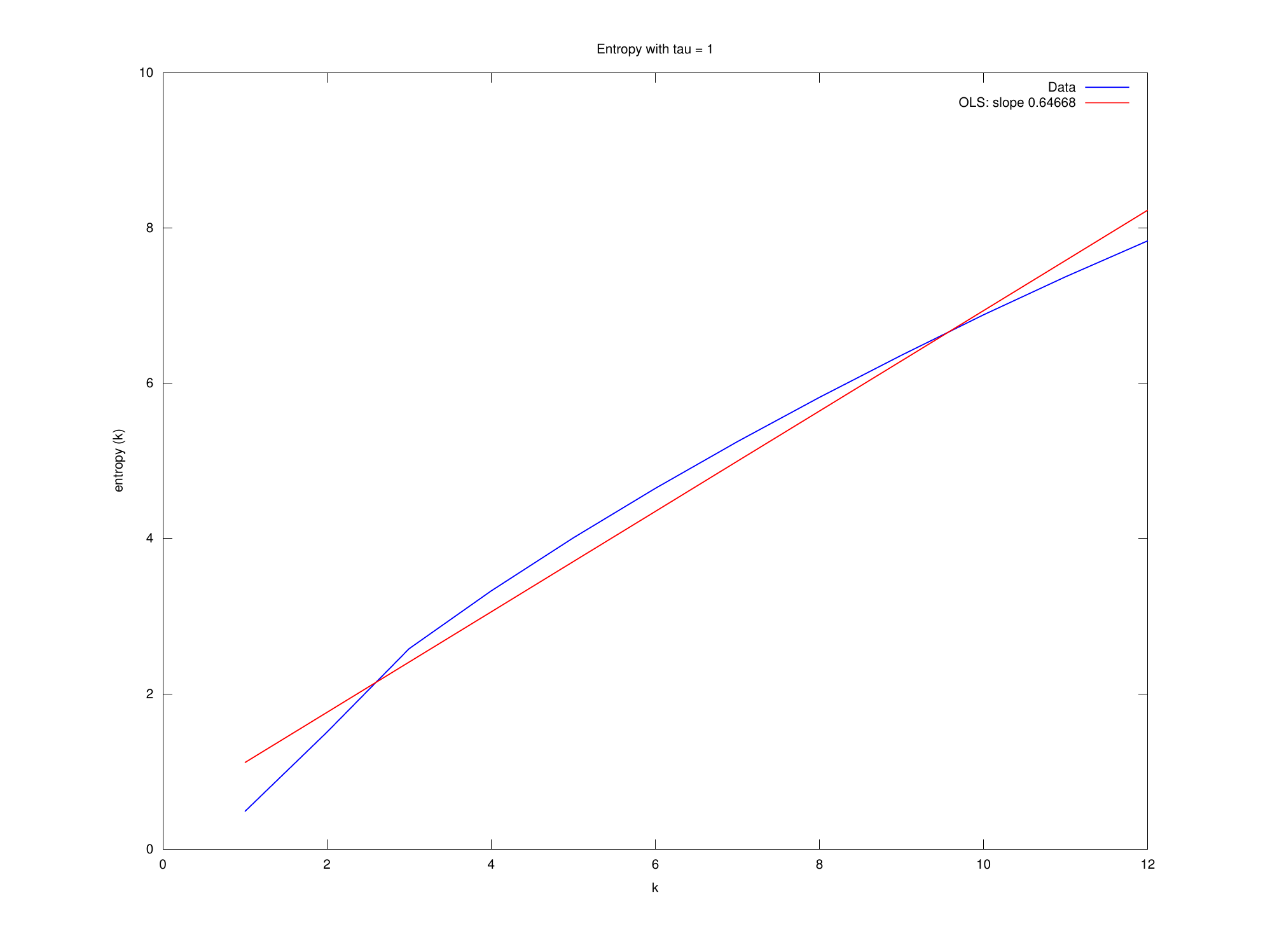}
\end{center}
\caption{\label{fig.HH.P-chaos}%
Analysis of chaos for $P$ in setting \HH. 
Top left: autocorrelation function. 
Its first (approximate) zero is at $\tau^{\star}=1.37$. 
Middle: sequence of the returns $r_k(\tau) = \log_{10} P((k+1) \tau)  - \log_{10} P(k \tau) $ 
for $\tau = \tau^{\star}$ (top) and $\tau=1$ (bottom). 
Right: estimated entropy $H_K$ as a function of $K$, 
with $\tau=\tau^{\star}$ (top; slope $0.611$ and correlation coefficient $0.995$) 
and $\tau=1$ (bottom; slope $0.647$ and correlation coefficient $0.983$). 
See Section~\ref{sec.results.HH.chaos} and Appendix~\ref{app.technical-details.chaos} for details.
}
\end{figure}

\subsection{Exploration of the parameter space}  \label{sec.results.explor}

Starting from setting \HH, we changed the parameters one by one 
between  $\gamma$ and $m_0$. 
The resulting bifurcation diagrams are shown on Figures~\ref{fig.bifurc.HH-gamma.1d.Nr} and \ref{fig.bifurc.HH-m0.1d.Nr}, 
for $N_r$ and
Figures~\ref{fig.bifurc.HH-gamma.1d.P} and \ref{fig.bifurc.HH-m0.1d.P}, for $P$. 
Appendix~\ref{app.technical-details.bifurcation} provides details about how the bifurcation diagrams have been obtained.


\subsubsection{Setting \HH\ with $\gamma$ varying} \label{sec.results.explor.HH-gamma}
A bifurcation diagram is plot on Figure~\ref{fig.bifurc.HH-gamma.1d.Nr}, 
showing $\set{ N_r(t) , t \in \N}$ as a function of $\gamma \in [2,10]$. 
A corresponding bifurcation diagram for the price $P$ is shown on 
Figure~\ref{fig.bifurc.HH-gamma.1d.P}. 
The dynamics look like complex for the largest part of the interval, with only few and very small windows of periodicity.
In particular, chaotic behavior seems to arise also for small values of $ \gamma $ (i.e. when the fertility rate $ m(N) $ depends
in a weaker way on the population $N$) which are probably more reasonable in a human-controlled breeding facility.
\begin{figure}
\begin{center}
\includegraphics[width=0.8\textwidth]{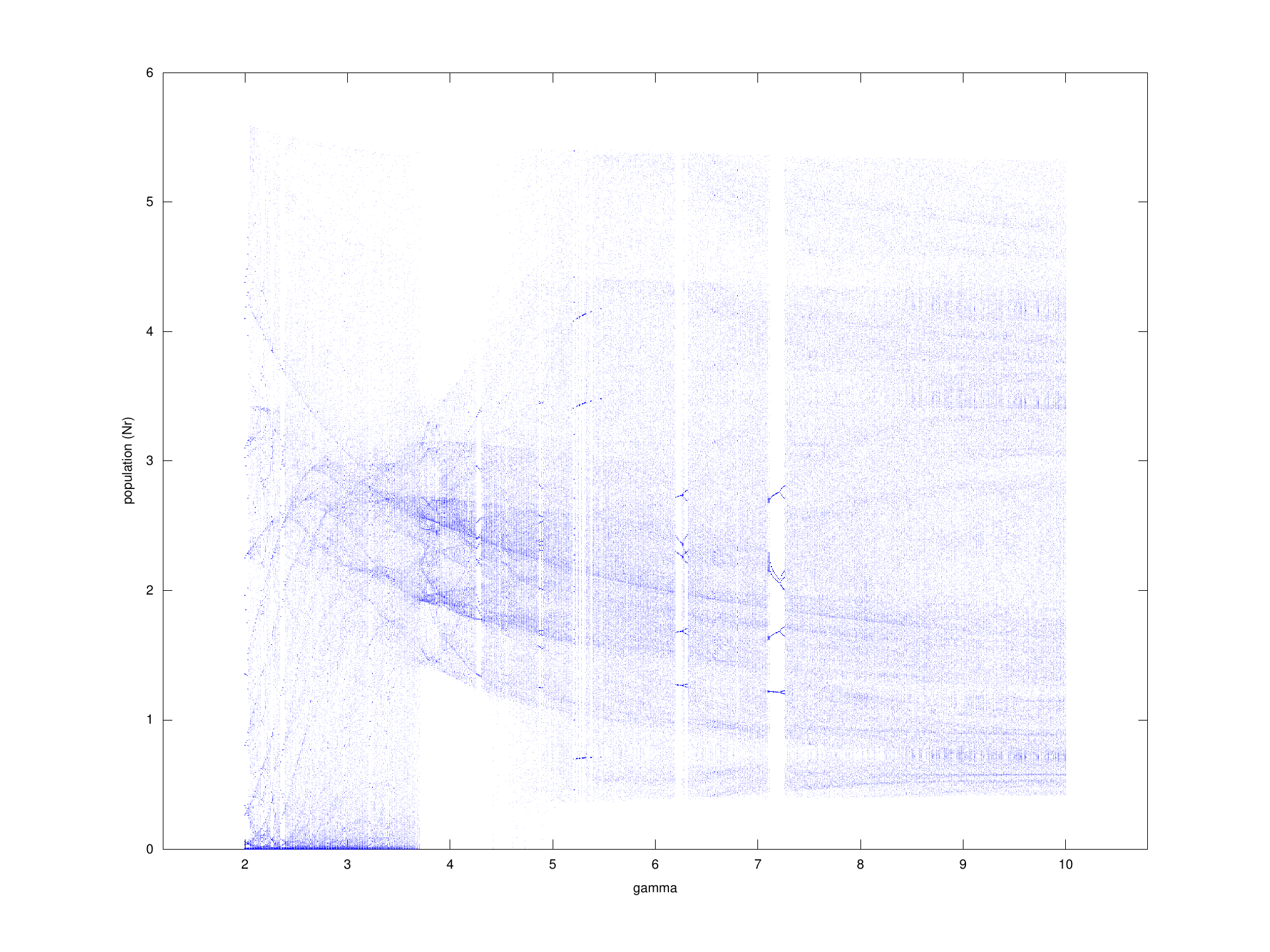}
\end{center}
\caption{\label{fig.bifurc.HH-gamma.1d.Nr}Bifurcation diagram for $N_r(t)$ w.r.t. $\gamma \in [2,10]$ around setting \HH. 
}
\end{figure}
\begin{figure}
\begin{center}
\includegraphics[width=0.8\textwidth]{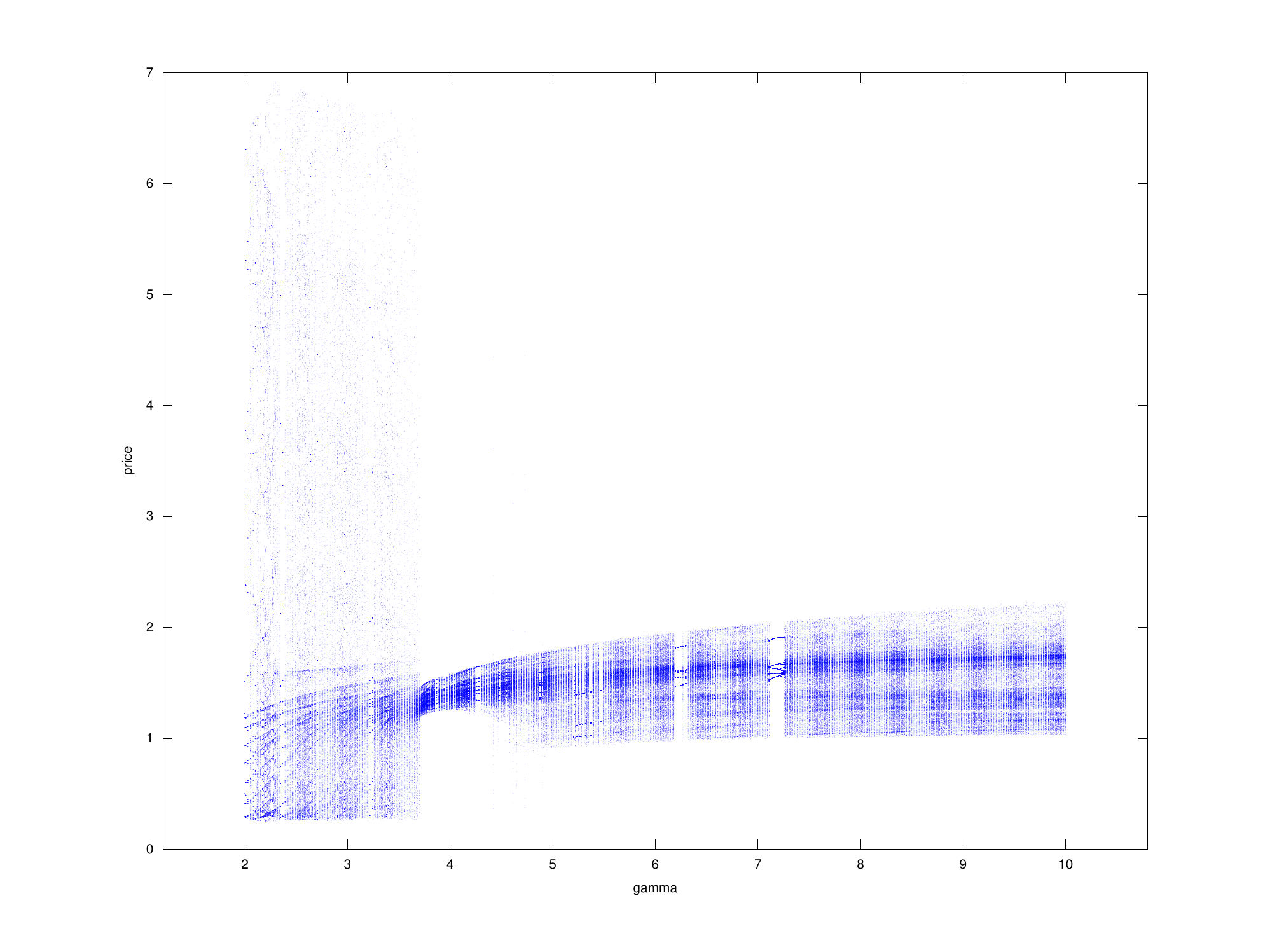}
\end{center}
\caption{\label{fig.bifurc.HH-gamma.1d.P}Bifurcation diagram for $P(t)$ w.r.t. $\gamma \in [2,10]$ around setting \HH. 
}
\end{figure}

\subsubsection{Setting \HH\ with $m_0$ varying} \label{sec.results.explor.HH-m0}
A bifurcation diagram is plot on Figure~\ref{fig.bifurc.HH-m0.1d.Nr}, showing $\set{ N_r(t) , t \in \N}$ as a function of $m_0 \in [2,8]$. 
The corresponding bifurcation diagram for the price $P$ is shown on Figure~\ref{fig.bifurc.HH-m0.1d.P}.
Here period doubling cascades seem to arise in the windows of periodicity.
Moreover, at a closer inspection of Figures~\ref{fig.bifurc.HH-m0.1d.Nr} and~\ref{fig.bifurc.HH-m0.1d.P}, a relatively small amplitude Hopf bifurcation
seems to happen near $ m_{0} \approx 5.4 $. 
\begin{figure}
\begin{center}
\includegraphics[width=0.8\textwidth]{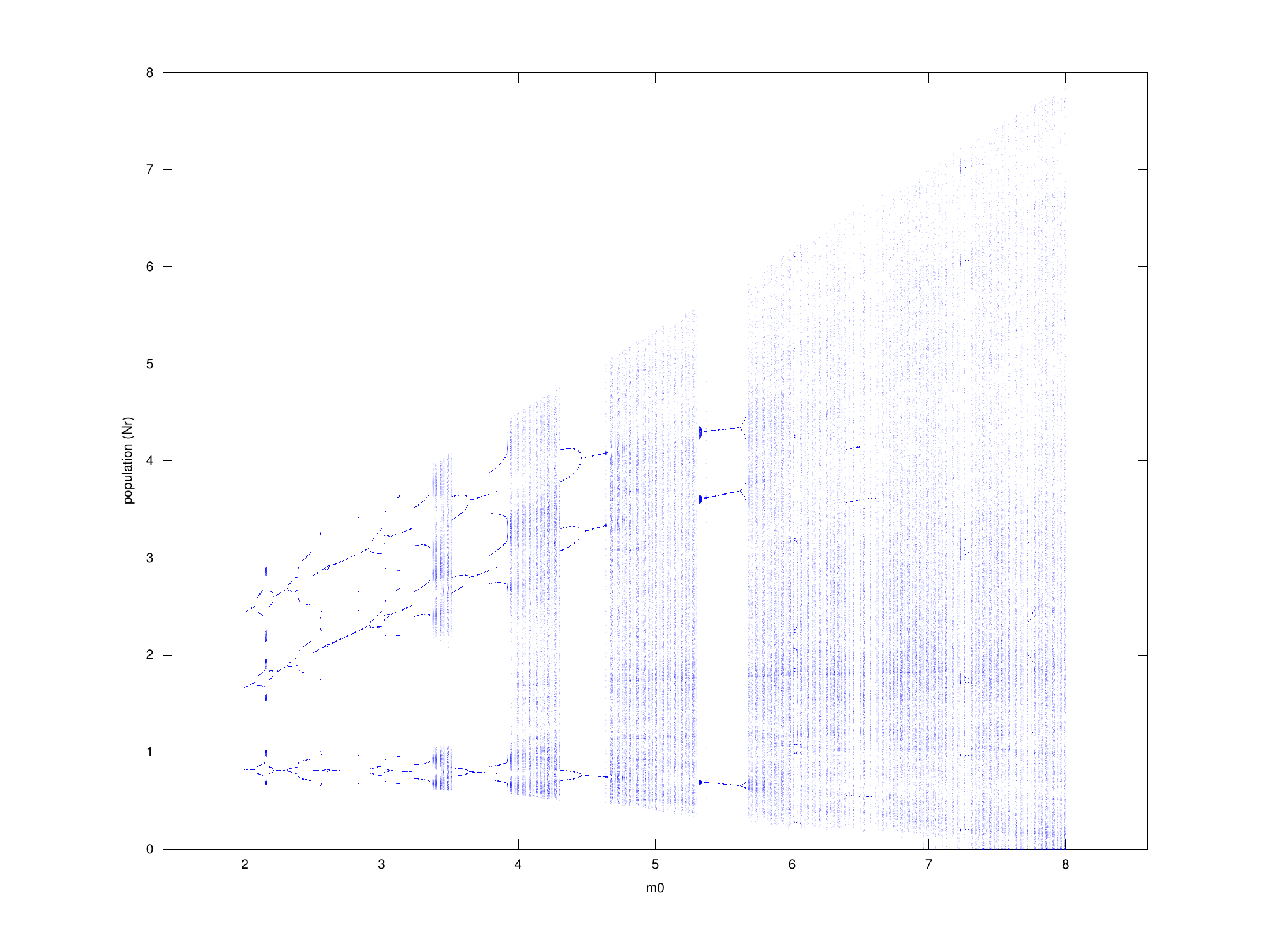}
\end{center}
\caption{\label{fig.bifurc.HH-m0.1d.Nr}Bifurcation diagram for $N_r(t)$ w.r.t. $m_0 \in [2,8]$ around setting \HH. 
}
\end{figure}
\begin{figure}
\begin{center}
\includegraphics[width=0.8\textwidth]{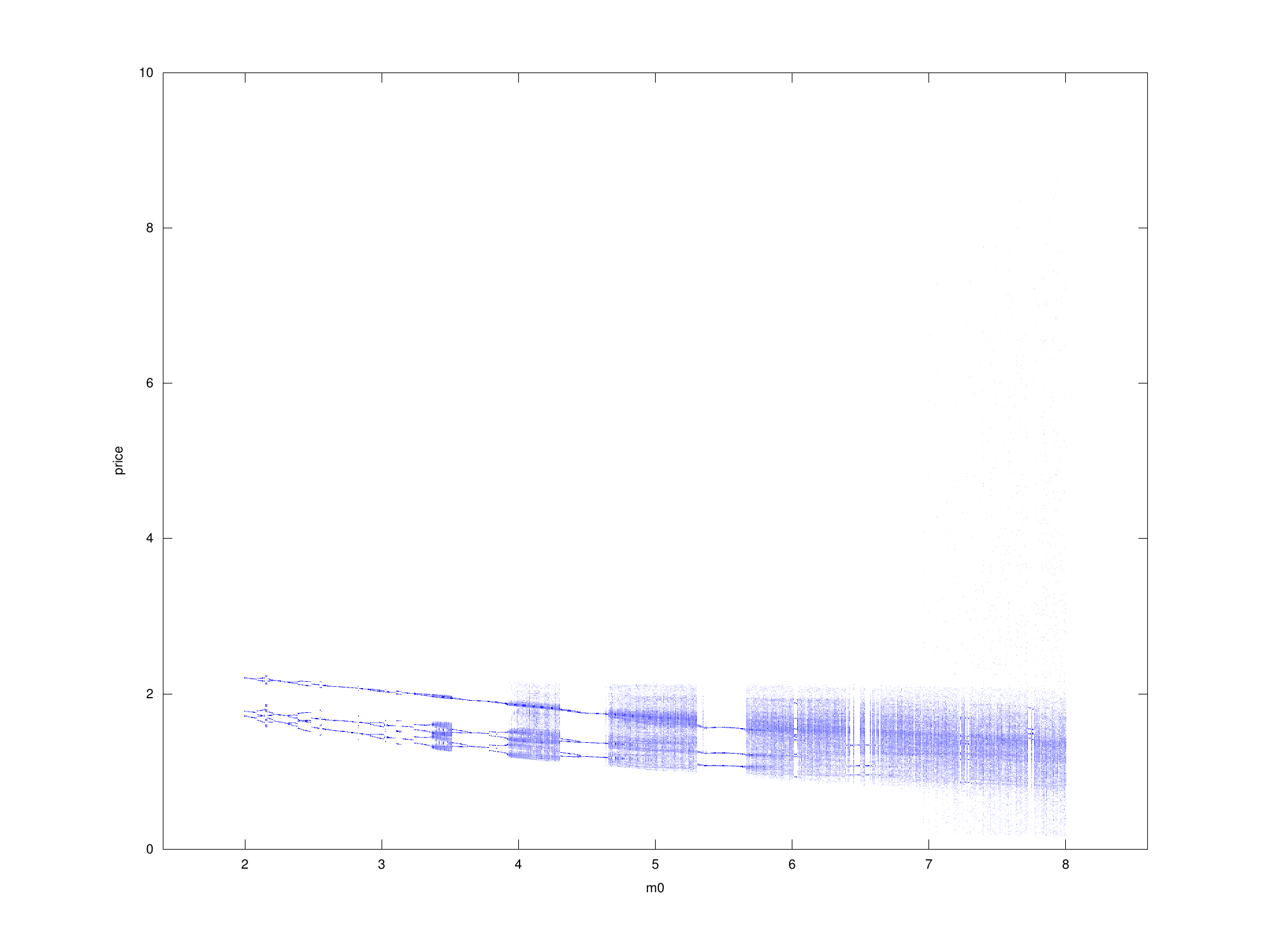}
\end{center}
\caption{\label{fig.bifurc.HH-m0.1d.P}Bifurcation diagram for $P(t)$ w.r.t. $m_0 \in [2,8]$ around setting \HH. 
}
\end{figure}

\bibliographystyle{alpha}

\appendix

\section{Technical details about the simulations} \label{app.technical-details}

This section provides all technical details necessary to reproduce our numerical experiments. 

\subsection{Discretization of the model} \label{app.technical-details.discretization}

Throughout the paper, we consider the following discretization with $q=100$ steps per year.
This will lead to a deterministic dynamical system in a phase space of dimension $ 2 \times 201 $.

\subsubsection{Notation and parameters}

We replace the continous time parameter $t \in [0,+\infty)$ by indices $i \in \N\backslash\set{0}$. \\
Roughly, $i \geq 1$ replaces the interval $\left( \frac{i-1}{q} , \frac{i}{q} \right]$ 
where $q \geq 1$ is an integer (number of steps per year). 

\begin{itemize}
\item $q \in \N\backslash\set{0}$ is the number of steps per year
\item $N_{r,i} $ approximates the \emph{average} of $N_r(t)$ over $t \in \left( \frac{i-1}{q} , \frac{i}{q} \right]$: 
\[ N_{r,i} \approx q \int_{\frac{i-1}{q}}^{\frac{i}{q}} N_r(t) dt \]
\item Similarly, $N_{b,i}$, $\rhoseasondiscrete{i}$, $S_i$ and $P_i$ are respectively defined 
from $N_{b}(t)$, $\rhoseason (t)$, $S(t)$ and $P(t)$. 
\item $B_{r,i}$ approximates the \emph{number} of newborn females put in the reproducing line 
in the interval $t \in \left( \frac{i-1}{q} , \frac{i}{q} \right]$: 
\[ B_{r,i} \approx \int_{\frac{i-1}{q}}^{\frac{i}{q}} B_r(t) dt  \]
Note the difference with $N_r$ coming from the fact $B_r(t)$ is a density (number per unit of time) 
contrary to $N_r(t)$ that is a number of individuals. 
\\
Similarly, $B_{b,i}$, $B_{f,i}$, $B_{m,i}$ are defined respectively from $B_{b}(t)$, $B_{f}(t)$, $B_{m}(t)$.
\end{itemize}

We also need to discretize some of the model parameters (with a few technical modifications in order to avoid some troubles):
\begin{itemize}
\item $k_{A_0}=\max\set{1, \croch{q A_0} }$ where $\croch{x}$ is the closest integer to $x$. 
We want $k_{A_0} >0$ so that $N_{r,k}$ only depends 
on the past of $B_r$ in \eqref{eq.modelB-discrete.Nr}, 
otherwise we would have a circular definition.
\item $k_{A_1}=\max\set{k_{A_0},\croch{q A_1}-1}$. 
We want $k_{A_1}\geq k_{A_0}$ in order to make the sum \eqref{eq.modelB-discrete.Nr} non-empty when $A_0=A_1$. 
We take $\croch{q A_1}-1$ instead of $\croch{q A_1}$ because animals are supposed to die exactly at age $A_1$, 
so individual of age $A_1$ at time $k/q$ (counted in $B_{r , k-\croch{q A_1}}$) do not count in $N_{r,k}$. 
\item $k_{\Omega_0}=\max\set{1, \croch{q \Omega_0} }$ (similarly to $k_{A_0}$)
\item $k_{\Omega_1}=\max\set{k_{\Omega_0},\croch{q \Omega_1}-1}$ (similarly to $k_{A_1}$)
\end{itemize}

\subsubsection{Discretized model}

At any time step $k \geq 1$, we compute 
\[ Z_k \egaldef (N_{r,k} \, ; \, N_{b,k} \, ; \, S_k \, ; \, P_k \, ; \, B_{r,k} \, ; \, B_{b,k}) \] 
from its past values 
\[ \paren{Z_i}_{ k - \max\{k_{A_1},k_{\Omega_1} \} \leq i \leq k-1} \enspace ,\]
in the following order:
\begin{align} \label{eq.modelB-discrete.Nr}
N_{r,k} &= \sum_{j= k_{A_0}}^{k_{A_1}} B_{r,k-j} \\ \label{eq.modelB-discrete.Nb}
N_{b,k} &= \sum_{j= k_{\Omega_0}}^{k_{\Omega_1}} B_{b,k-j} \\ \label{eq.modelB-discrete.S}
S_k &= \frac{q N_{b,k}}{k_{\Omega_1}-k_{\Omega_0}+1} \\ \label{eq.modelB-discrete.P}
P_k &= \max\set{ 0, P_{k-1} + \frac{\lambda P_{k-1} F \bigl( D(P_{k-1}),S_k \bigr)}{q} }
\\ \label{eq.modelB-discrete.Br}
B_{r,k} &= \frac{m_0}{q} \rhoseasondiscrete{k} m(N_{r,k}) R(P_k) 
\\ \label{eq.modelB-discrete.Bb}
B_{b,k} &= \frac{m_0}{q} \rhoseasondiscrete{k} m(N_{r,k}) \parenb{2 - R(P_k)}
\end{align}

\subsection{Initial condition} \label{app.technical-details.init}
The initial condition is given by the values of the number of newborns ($B_{r,i}$ and $B_{b,i}$) 
for time steps $i=1 , \ldots , k_{A_1}$ before the start of the simulation, 
since the discretized model above (Eq.~\eqref{eq.modelB-discrete.Nr}--\eqref{eq.modelB-discrete.Bb}) 
only needs the past values of $B_r$ and $B_b$ to compute the future dynamics of the model. 

For each parameter set, we choose an initial condition randomly (but with a known random seed, 
so we can use repeatedly the same initial condition to obtain our bifurcation diagrams, for instance) as follows: 
$ B_{r,1} ,\ldots, B_{r,k_{A_1}} , B_{b,1} , \ldots, B_{b,k_{A_1}} $  are chosen independent 
with common distribution $ \mathcal{U}\paren{ \croch{ 0 , \frac{2}{k_{A_1} - k_{A_0} + 1} }} $.
The reason for this choice is that $k_{A_1} - k_{A_0} + 1$ is the number of time steps 
corresponding to reproducing ages of females. 
So, for instance, by Eq.~\eqref{eq.modelB-discrete.Nr}, the reproducing female population at time $k_{A_1}+1$ 
is the average of random variables uniform over $[0,2]$, so it should be close to 1, 
the threshold value for density-dependence to apply. 

\subsection{Chaos analysis} \label{app.technical-details.chaos}
In each setting and for each time series $Y \in \set{N_r,P}$, 
we start from the ``continuous'' time series of $Y(t)$ for $t>T_{\max}-10\,000$, 
where $T_{\max}$ is the total length of the simulation experiment 
(i.e., $T_{\max}=300\,000$ for setting \HH, and $T_{\max}=100\,000$ for other settings). 
Then, we compute the empirical autocorrelation function of $Y$ with lags $\tau \in [0,100]$ 
($\tau$ is expressed in years). 
We get its first ``zero'' $\tau^{\star}$ as the first point where the empirical autocorrelation crosses zero, 
and we then check that the value of the empirical autocorrelation at $\tau=\tau^{\star}$ 
is smaller than an arbitrary threshold (here, $10^{-2}$); 
this holds true in all the settings where we did such an analysis. 

Then, we consider the sequence $Y(t)$ only at times $t$ equal to an integer multiple of $\tau^{\star}$, 
and compute the returns $\log_{10} (Y((i+1)\tau^{\star}) / Y( i \tau^{\star}))$ 
and their respective signs $\varepsilon_i \in \set{-1,1}$. 
From this (finite) sequence of signs, 
we compute for $K=1, \ldots, 12$ the combinatorial entropy $H_K (\tau^{\star})$ of 
$( (\varepsilon_i)_{k-K+1 \leq i \leq k} )_{k \in \N}$:
\[ 
H_K (\tau^{\star}) \egaldef - \sum_{ x \in \set{-1,1}^K } p_x \log_2(p_x) 
\qquad \mbox{where} \quad 
p_x = \Proba \parenb{ (\varepsilon_i)_{k-K+1 \leq i \leq k} = x} 
\]
the latter probability being with respect to $k$. 

Finally, we plot $H_K (\tau^{\star})$ as a function of $K$ and 
perform a standard robust linear regression in order to estimate its slope.

We also do the same analysis when considering $Y(t)$ at integer times, 
leading to a plot of $H_K (1)$ as a function of $K$, 
on which we estimate its slope by performing a standard robust linear regression. 
\begin{figure}
\begin{center}
\includegraphics[width=0.8\textwidth]{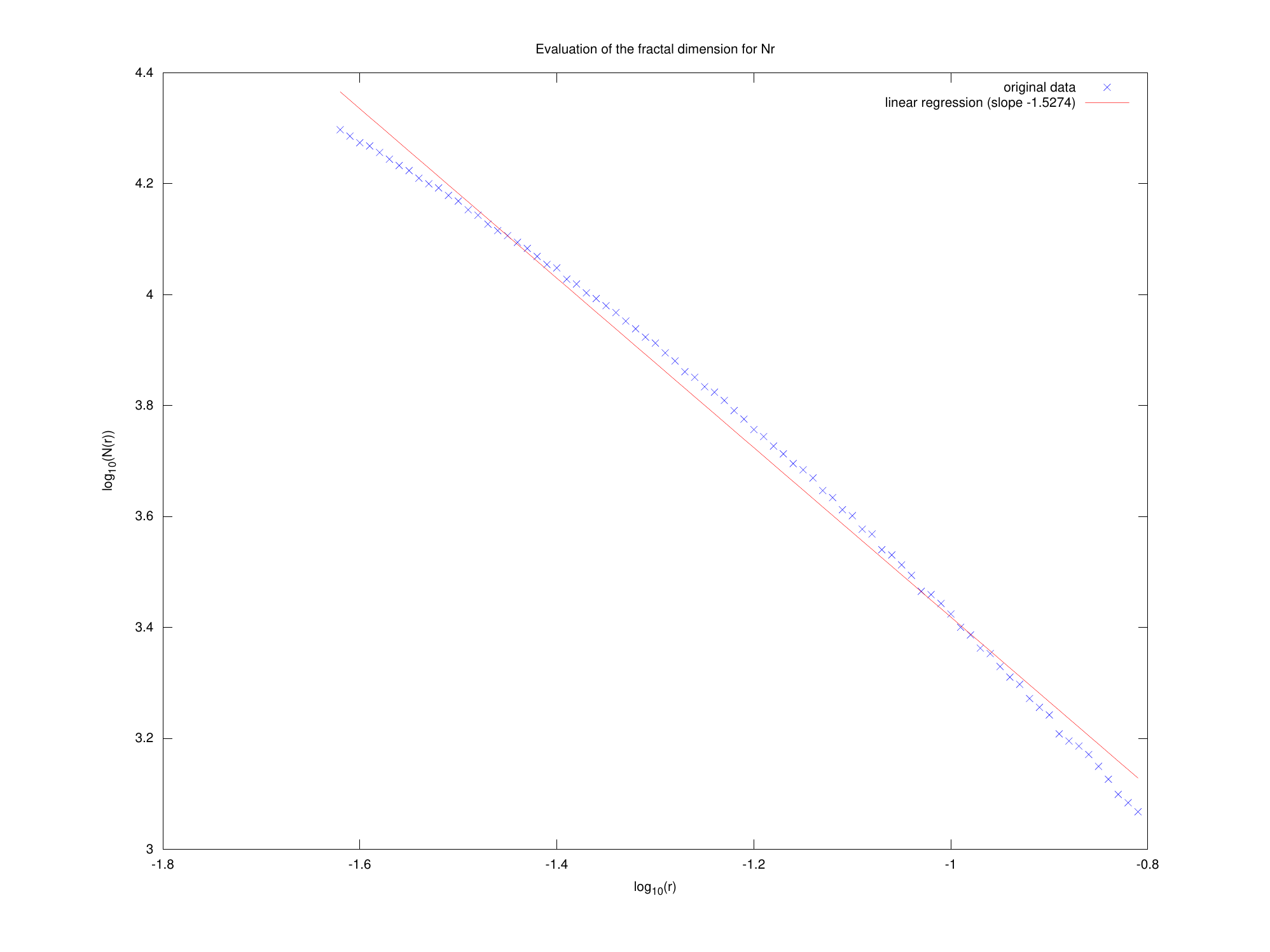}
\end{center}
\caption{\label{fig.HH.Nr-dim_f}
Estimation of the fractal dimension of the attractor of Figure~\ref{fig.HH.Nr-3d.less-pts} (setting \HH, $N_r$).}
\end{figure}
\begin{figure}
\begin{center}
\includegraphics[width=0.8\textwidth]{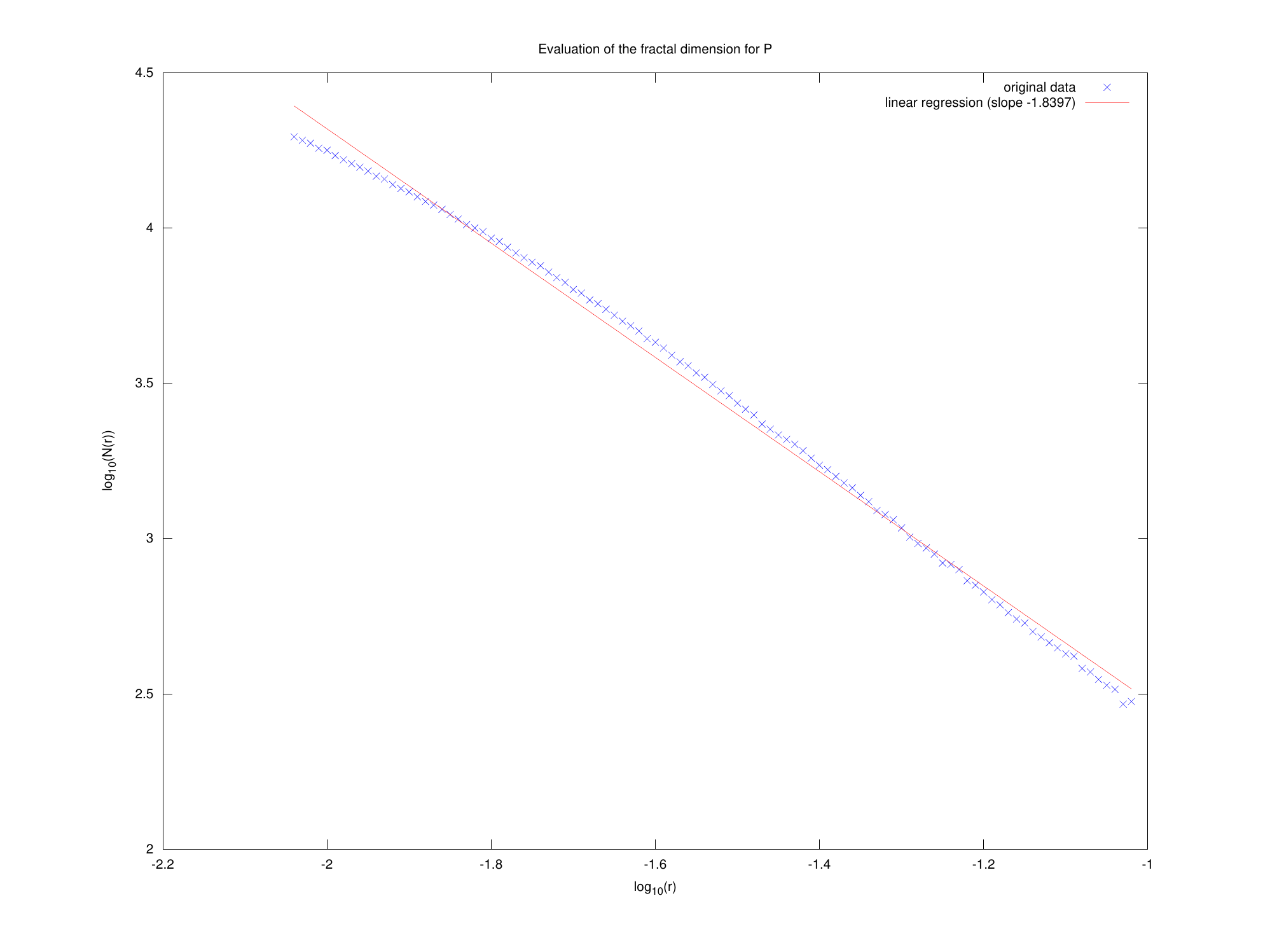}
\end{center}
\caption{\label{fig.HH.P-dim_f}
Estimation of the fractal dimension of the attractor of Figure~\ref{fig.HH.P-3d.less-pts} (setting \HH, $P$).}
\end{figure}

\subsection{Fractal dimension} \label{app.technical-details.dim_f}
For computing the fractal dimension of the attractors for setting \HH, 
we started from their 3-dimensional visualization over $200\,000$ years, that is, for $N_r$ for instance: 
\[ \cK = \Bigl\{ \bigl( N_r(t),N_r(t+1), N_r(t+2) \bigr)  \, , \, 100\,000 \leq t \leq 300\,000 \Bigr\} \enspace . \]

Then, for various values of $\epsilon>0$, we compute the number 
$\widetilde{\cN}_{\epsilon}(\cK)$ of cubes $\mathcal{C}_{i,j,k} = [i \epsilon;
(i+1) \epsilon] \times [j \epsilon; (j+1) \epsilon] \times [k
\epsilon; (k+1) \epsilon]$ that contain at least one point of $\cK$. 
Figure~\ref{fig.HH.Nr-dim_f} thus represents $\log_{10} \widetilde{\cN}_{\epsilon}(\cK)$ as a function of 
$\log_{10} \epsilon$.

Theoretically, the fractal dimension is the opposite of the slope of this graph at infinity. 
Here, since $\cK$ is finite, $\widetilde{\cN}_{\epsilon}(\cK)$ is constant equal to $\card(\cK)$ for small $\epsilon$. 
So, we estimated the slope of the graph only for a limited set of values of $\epsilon$. 
One can check on Figure~\ref{fig.HH.Nr-dim_f} that the linear fit seems reasonably close to the original curve. 

Because of the numerous approximations made during this estimation, 
the precise value of the fractal dimension should not be taken into account too seriously, 
but its order of magnitude should be correct. 
Figure~\ref{fig.HH.P-dim_f} shows the corresponding estimation with $P$ instead of $N_r$. 

\subsection{Bifurcation diagrams} \label{app.technical-details.bifurcation}
We took the same initial condition for all settings in all diagrams (by choosing a fixed arbitrary seed to the random number generator). 
Then, for each value of the parameter we plot the sequence of values of $N_r(t)$ (resp. $P(t)$) for $t \in \set{ 1\,500 , \ldots, 2\,000 }$. 
In all diagrams, the parameters are varying by step of $0.01$.

\end{document}